\documentclass[final,onefignum,onetabnum]{siamart220329}

\usepackage{braket,amsfonts}

\usepackage{array}

\usepackage[caption=false]{subfig}

\usepackage{pgfplots}

\newsiamthm{claim}{Claim}
\newsiamremark{remark}{Remark}
\newsiamremark{hypothesis}{Hypothesis}
\crefname{hypothesis}{Hypothesis}{Hypotheses}

\usepackage{algorithmic}

\usepackage{graphicx,epstopdf}

\Crefname{ALC@unique}{Line}{Lines}

\usepackage{amsopn}

\usepackage{xspace}
\usepackage{bold-extra}
\usepackage[most]{tcolorbox}
\usepackage{showkeys}

\colorlet{texcscolor}{blue!50!black}
\colorlet{texemcolor}{red!70!black}
\colorlet{texpreamble}{red!70!black}
\colorlet{codebackground}{black!25!white!25}

\lstdefinestyle{siamlatex}{%
  style=tcblatex,
  texcsstyle=*\color{texcscolor},
  texcsstyle=[2]\color{texemcolor},
  keywordstyle=[2]\color{texemcolor},
  moretexcs={cref,Cref,maketitle,mathcal,text,headers,email,url},
}

\tcbset{%
  colframe=black!75!white!75,
  coltitle=white,
  colback=codebackground, 
  colbacklower=white, 
  fonttitle=\bfseries,
  arc=0pt,outer arc=0pt,
  top=1pt,bottom=1pt,left=1mm,right=1mm,middle=1mm,boxsep=1mm,
  leftrule=0.3mm,rightrule=0.3mm,toprule=0.3mm,bottomrule=0.3mm,
  listing options={style=siamlatex}
}

\newtcblisting[use counter=example]{example}[2][]{%
  title={Example~\thetcbcounter: #2},#1}

\newtcbinputlisting[use counter=example]{\examplefile}[3][]{%
  title={Example~\thetcbcounter: #2},listing file={#3},#1}

\DeclareTotalTCBox{\code}{ v O{} }
{ 
  fontupper=\ttfamily\color{black},
  nobeforeafter,
  tcbox raise base,
  colback=codebackground,colframe=white,
  top=0pt,bottom=0pt,left=0mm,right=0mm,
  leftrule=0pt,rightrule=0pt,toprule=0mm,bottomrule=0mm,
  boxsep=0.5mm,
  #2}{#1}

\patchcmd\newpage{\vfil}{}{}{}
\begin{tcbverbatimwrite}{tmp_\jobname_header.tex}
\title{Asymptotic-preserving schemes for the initial-boundary value problem of hyperbolic relaxation systems 
}

\author{Yizhou Zhou \thanks{IGPM, RWTH Aachen University, Templergraben, 55, D-52062 Aachen, Germany (zhou@igpm.rwth-aachen.de)}}

\headers{AP schemes for the IBVP of hyperbolic relaxation systems}{Yizhou Zhou}
\end{tcbverbatimwrite}
\input{tmp_\jobname_header.tex}

\ifpdf
\hypersetup{ pdftitle={AP schemes for the IBVP of hyperbolic relaxation systems} }
\fi

\begin{document}
\maketitle
\begin{tcbverbatimwrite}{tmp_\jobname_abstract.tex}
\begin{abstract}
In this work, we present a numerical method for the initial-boundary value problem (IBVP) of first-order hyperbolic systems with source terms. The scheme directly solves the relaxation system using a relatively coarse mesh and captures the equilibrium behavior quite well, even in the presence of boundary layers. This method extends the concept of asymptotic-preserving schemes from initial-value problems to IBVPs. Moreover, we apply this idea to design a unified numerical scheme for the interface problem of relaxation systems.

\end{abstract}


\end{tcbverbatimwrite}
\input{tmp_\jobname_abstract.tex}

\section{Introduction}

This paper is concerned with the numerical schemes for the initial-boundary value problem (IBVP) of the  first-order hyperbolic relaxation system 
\begin{eqnarray*}
\left\{
\begin{array}{l}
     U_t + A(U) U_x = \dfrac{1}{\epsilon}Q(U),\quad x>0,~t>0,\\[2mm]
    BU(0,t)= b(t),\\[2mm]
    U(x,0)= g(x).
\end{array}\right.
\end{eqnarray*}
Here $A(U)$ and $Q(U)$ are smooth functions of unknown $U=U(x,t)\in \mathbb{R}^n$, $B$ is a constant matrix, $b(t)$ and $g(x)$ represent the boundary and initial data, and $\epsilon$ is a small parameter called the relaxation time. This kind of problems occur naturally in the kinetic theory \cite{MR3394370,MR416379,MR1392419,MR1774264}, chemically reactive flows \cite{MR1618466}, non-equilibrium thermodynamics \cite{MR2573593,muller2007history,zhu2015conservation}, traffic flows \cite{MR3816765,MR3867627,MR4656780}, and so on.

As $\epsilon$ goes to zero, the relaxation system formally converges to the equilibrium system.
The rigorous proof of relaxation limit has been given in \cite{MR1693210} for the initial value problem (IVP) under the structural stability condition. For the IBVPs, the theoretical results have been developed under the stability condition and a so-called generalized Kreiss condition (GKC) \cite{MR1722195}. Particularly, the reduced boundary condition for the equilibrium system was derived from the boundary condition of the relaxation system. The interested reader is referred to \cite{MR1604274,MR1793199,MR1900493,MR1919787,MR2030150,MR4355918} for further works in this direction.

In order to solve the partial differential equations with stiff source term, a so-called asymptotic-preserving (AP) concept has been proposed for the effective numerical schemes. As pointed out in \cite{MR2964096}, the basic idea of AP scheme is to develop numerical methods that preserve the asymptotic limits from the microscopic to the macroscopic models, in the discrete setting. This kind of numerical schemes have been well-developed for the IVPs of relaxation systems and kinetic equations, see \cite{MR1445737,MR3622624,MR2674294,MR3645390,MR4436589,MR1655853,MR888058,MR2460781,MR1619859} and the references cited therein. Particularly, the macroscopic behavior can be captured by solving the original relaxation system with coarse grids, even in cases where the relaxation system admits an initial layer \cite{MR1358521}.

Nevertheless, there are rare paper about the numerical methods for the IBVPs of relaxation systems. In this work, we investigate the possibility to capture the macroscopic behavior of IBVPs by using coarse grids, particularly in cases where the relaxation system features boundary layers. To this end, we need to consider two questions: (1) Given a numerical scheme, how to check the AP property for the IBVPs? (2) How to design an AP scheme for the IBVPs?

To address the first question, we recall the definition of AP scheme \cite{MR2964096}. 
Denote the model of relaxation system by $\mathcal{F}^{\epsilon}$. As $\epsilon\rightarrow 0$ the model is approximated by an equilibrium model $\mathcal{F}^{0}$ which is independent of $\epsilon$. Denote the numerical scheme for $\mathcal{F}^{\epsilon}$ by $\mathcal{F}^{\epsilon}_{\delta}$ with $\delta$ the mesh size. Then the scheme is said to be AP if the limiting scheme $\mathcal{F}^{0}_{\delta}$ is a good approximation of $\mathcal{F}^{0}$.
To check the AP property for IVPs, one can usually take $\epsilon\rightarrow 0$ in $\mathcal{F}^{\epsilon}_{\delta}$
and see if the limiting scheme
$\mathcal{F}^{0}_{\delta}$ is consistent with the equilibrium system $\mathcal{F}^{0}$.

However, this step is not so straightforward for the IBVPs. 
We show the AP property of $\mathcal{F}^{\epsilon}_{\delta}$ by the following way. Denote the solution of the relaxation system
$\mathcal{F}^{\epsilon}$ by $U^{\epsilon}$. The relaxation limit of $U^{\epsilon}$ as $\epsilon\rightarrow 0$, if exists, is denoted by $U^0$. One can substitute this limit into the limiting scheme $\mathcal{F}^{0}_{\delta}$ and check the truncation error. In general, it is not easy to express the original solution $U^\epsilon$ directly. Therefore, we exploit the asymptotic solution $U_\epsilon$ in \cite{MR1793199,MR1722195}, take the limit $U_\epsilon\rightarrow U_0$, and substitute $U_0$ into the limiting scheme $\mathcal{F}^{0}_{\delta}$.
Thanks to the error estimates \cite{MR1793199,MR1722195}, the difference between $U_\epsilon$ and $U^\epsilon$ is small and thus it is reasonable to use the limit $U_0$ as a replacement.

Now we turn to discuss the second question. Note that the task is not trivial even for the first-order scheme due to the effect of boundary layers. In Section \ref{Section2}, we show this by considering a simple relaxation system---the Jin-Xin model. 
It is well-known that, as $\epsilon\rightarrow 0$, the first-order upwind implicit-explicit (IMEX) scheme for the Jin-Xin model becomes the Lax-Friedrichs scheme for the equilibrium system. However, in considering IBVPs, the usual discretization for the boundary condition may lead to wrong boundary value due to the existence of boundary layer (see Section \ref{section2.2}). Moreover, the error of the boundary value may spread and destroy the solution in the whole domain (see Example 4 in Section \ref{Section5}).

In this work, we give a new scheme by resorting to the theoretical result in \cite{MR1722195}.
Specifically, we adapt the matrix in GKC to design a new spatial discretization for the flux term. 
For the non-stiff case $\epsilon\gg \delta$, our scheme reduces to the normal first-order upwind scheme. For the stiff-case $\epsilon\ll \delta$, our scheme automatically gives a proper discretization for the equilibrium system. According to the asymptotic analysis in \cite{MR1722195}, the thickness of boundary-layer should be
$O(\epsilon)$. It means that the boundary layer behaves like a jump  when we take the mesh size  $\delta\gg\epsilon$. Surprisingly,  this jump can be well-captured by our scheme with coarse mesh. In Section \ref{Section2}, we first consider the Jin-Xin model and illustrate the basic idea to design boundary AP scheme. Then we extend the idea to IBVPs of general linear relaxation systems with constant coefficients \cite{MR1722195}. For simplicity, we assume that there is no initial-layer and the boundary is non-characteristic for both the relaxation system and equilibrium system \cite{MR1722195}. 

Furthermore, we exploit the idea to design a scheme for the interface problem. Namely, we consider the following IVP for the relaxation system:
\begin{align*}
    &U_t+A(U)U_x=\frac{1}{\epsilon(x)}Q(U),\qquad x\in \mathbb{R}
\end{align*}
and the relaxation parameter has a discontinuity at $x=0$, i.e., 
$\epsilon(x)=1$ for $x<0$ and
$\epsilon(x)=\epsilon_0\ll1$ for $x\geq0$.
This type of multiscale problems describe many important physical phenomena such as the space shuttle reentry problems in aerodynamics \cite{MR1262639,MR1335825}, the transport problems for materials with very different
opacities \cite{MR1916293,MR2026400}. In these problems, the relaxation time varies drastically near the interface.
It was shown in \cite{MR4593212,MR3008837} that the interface problem has very similar features with the IBVP. Particularly, there exists interface layer near the point $x=0$. For such problems, one can either use very fine mesh ($\delta \ll \epsilon_0$) or use the domain decomposition method \cite{MR533346,MR4593212,MR3008837} which solves different equations 
on different regions according to regimes of $\epsilon$. In this work, we develop a numerical method that solves one system (the relaxation system) with a unified scheme for the whole domain.

The paper is organized as follows. In Section \ref{Section2}, we show the formal expansion for the Jin-Xin model and illustrate the idea of asymptotic-preserving scheme for the IBVPs.
Section \ref{Section3} is devoted to present the AP scheme for the IBVPs of general linear relaxation systems. In Section \ref{Section4}, we apply the idea to design a unified numerical scheme for the interface problem. At last, in Section \ref{Section5}, we
give some numerical examples to show the validity of our scheme. Some missing details can be found in Appendix \ref{AppendA}.


\section{An illustrative example: Jin-Xin model}\label{Section2}

\subsection{Initial-boundary value problem}\label{sec:2.1}

Consider the following initial-boundary value problem for the Jin-Xin relaxation model:
\begin{align}
&\left\{\begin{array}{l}
u_t + v_x = 0,\\[3mm]
v_t + u_x = (f(u)-v)/\epsilon,
\end{array}
\right. \qquad x>0,~t>0, \label{JX-sys}\\[2mm]
&~~B_u u(0,t) + B_v v(0,t) = b(t).\label{JX-BC}
\end{align}
Here $u=u(x,t)$ and $v=v(x,t)$ are the unknowns, $f(u)$ is a convex function, $\epsilon$ is the relaxation time, and
$B_u$, $B_v$ are two constants. In this paper, we only consider the case where $f'(u)<0$ for any $u$ or $f'(u)>0$ for any $u$.

According to the asymptotic analysis in \cite{MR1793199,MR1722195}, we consider the ansatz 
\begin{equation}\label{JX-asy}
\begin{pmatrix}
    u_{\epsilon}\\
    v_{\epsilon}
\end{pmatrix}(x,t)=
\begin{pmatrix}
    \bar{u}\\
    \bar{v}
\end{pmatrix}(x,t)+
\begin{pmatrix}
    \mu\\
    \nu
\end{pmatrix}(y,t),\qquad y=\frac{x}{\epsilon}.
\end{equation}
Here $(\bar{u},\bar{v})$ is the outer solution satisfying $\bar{v} = f(\bar{u})$ and the equilibrium system
\begin{align}
\bar{u}_t + f(\bar{u})_x = 0.\label{JX-Eq1}
\end{align}
Besides, $(\mu,\nu)$ represents the boundary-layer correction satisfying $(\mu,\nu)(\infty,t)\equiv0$. 
To derive the equation for $(\mu,\nu)$, we introduce the operator 
$$
\mathcal{L}(u,v) := \begin{pmatrix}
u_t + v_x\\[1mm]
v_t + u_x - (f(u)-v)/\epsilon
\end{pmatrix}.
$$
Note that $(u_\epsilon,v_\epsilon)$ and $(\bar{u},\bar{v})$ satisfy the equation approximately. We compute
$$
\mathcal{L}(u_{\epsilon},v_{\epsilon})-\mathcal{L}(\bar{u},\bar{v}) = \begin{pmatrix}
\mu_t +\nu_y/\epsilon \\[2mm]
\nu_t + \mu_y/\epsilon - (f(\bar{u}+\mu)-f(\bar{u})-\nu)/\epsilon
\end{pmatrix}.
$$
Taking the coefficient of $O(1/\epsilon)$ to be zero, we have $\nu(y,t)\equiv\nu(\infty,t)=0$ and
\begin{equation}\label{JX-BL}
\mu_y = f(\bar{u}(0,t)+\mu) - f(\bar{u}(0,t)).
\end{equation}
Here we use the expansion
$\bar{u}(x,t)=\bar{u}(\epsilon y,t)=\bar{u}(0,t) + O(\epsilon)$. By the classical theory of ODE, when $f'(\bar{u}(0,t))>0$, the only bounded solution of \eqref{JX-BL} is $\mu\equiv0$. In this case, there is no boundary-layer. For the case $f'(\bar{u}(0,t))<0$, the equation \eqref{JX-BL} needs one initial condition at $y=0$.

Now we turn to consider the boundary condition. Since the asymptotic solution $(u_{\epsilon},v_{\epsilon})$ satisfies the boundary condition approximately, we have
$$
B_u(\bar{u}(0,t)+\mu(0,t)) + B_v f(\bar{u}(0,t)) = b(t).
$$
As $f'(\bar{u}(0,t))>0$, we know that $(\mu,\nu)\equiv 0$ and there is no boundary layer. The boundary condition for \eqref{JX-Eq1} is
\begin{equation}\label{JX-rBC}
B_u\bar{u}(0,t) + B_v f(\bar{u}(0,t)) = b(t).
\end{equation}
For the case $f'(\bar{u}(0,t))<0$, there is no boundary condition for \eqref{JX-Eq1} and the initial condition for \eqref{JX-BL} is given by
\begin{equation}\label{JX-mu}
B_u\mu(0,t) = b(t) - B_u\bar{u}(0,t) - B_v f(\bar{u}(0,t)).
\end{equation}

The error estimate in \cite{MR1793199} indicates that the asymptotic solution approximates the exact solution of \eqref{JX-sys}-\eqref{JX-BC} as $\epsilon\rightarrow 0$. On the other hand, we take $\epsilon\rightarrow 0$ in \eqref{JX-asy} to obtain the asymptotic limit
$$
(u_{lim},v_{lim})(x,t)=\left\{\begin{array}{lc}
\big(\bar{u}(x,t)+\mu(0,t),~f(\bar{u}(x,t))\big)\qquad  & f'(\bar{u}(0,t))<0, \\[2mm]
\big(\bar{u}(x,t),~f(\bar{u}(x,t))\big)  \qquad   & f'(\bar{u}(0,t))>0.
\end{array}\right.
$$ 
For $f'>0$, $\bar{u}$ is solved from the equilibrium system \eqref{JX-Eq1} with reduced boundary condition \eqref{JX-rBC}. For $f'<0$, $\bar{u}$ is solved from \eqref{JX-Eq1} and $\mu(0,t)$ is solved from \eqref{JX-mu}. As a preparation for the next discussions, we also denote the evaluation of $(u_{lim},v_{lim})$ on grid points $t_n=n \tau~(n\geq 0)$ and $x_j=jh ~(j\geq 0)$. For the case $f'<0$, we have  
\begin{equation}\label{JX-Alimit}
(u_{lim},v_{lim})(x_j,t_n)=\left\{\begin{array}{lc}
(\bar{u}_0^n+\mu_0^n,~f(\bar{u}_0^n))\qquad  & j=0, \\[2mm]
(\bar{u}_j^n,~f(\bar{u}_j^n))  \qquad   & j\geq 1,
\end{array}\right.
\end{equation}
where $\bar{u}_j^n=\bar{u}(x_j,t_n)$ and $\mu_0^n=\mu(0,t_n)$ are evaluation of $\bar{u}(x,t)$ and $\mu(0,t)$ on grid points. On the other hand, when $f'>0$, there is no boundary layer and we have
$
(u_{lim},v_{lim})(x_j,t_n)= 
(\bar{u}_j^n,~f(\bar{u}_j^n))$ for 
$j\geq 0$.

\subsection{Classical first-order upind IMEX scheme}\label{section2.2}
In this subsection, we show that the original upwind scheme may lead to wrong boundary value in computing the IBVPs.
Denote $U=(u,v)^T$ and the coefficient matrix of \eqref{JX-sys} by $A$. We have the characteristic decomposition
$$
A=R^A_+ \Lambda^A_+ L^A_+ +  R^A_- \Lambda^A_- L^A_-:=A_++A_-,
$$
where $\Lambda^A_+=1$ and $\Lambda^A_-=-1$ are the eigenvalues of $A$ and 
\begin{equation*}
R^A_+ = \frac{1}{\sqrt{2}} \begin{pmatrix}
    1 \\ 
    1
\end{pmatrix},\quad R^A_- = \frac{1}{\sqrt{2}} \begin{pmatrix}
    -1 \\ 
    1
\end{pmatrix},\quad L^A_+ = \frac{1}{\sqrt{2}} \begin{pmatrix}
    1 & 1
\end{pmatrix},\quad L^A_- = \frac{1}{\sqrt{2}} \begin{pmatrix}
    -1 & 1
\end{pmatrix}.
\end{equation*}
The first-order upwind IMEX scheme is:
\begin{align}
    &U^{n+1}_j-U^n_j + \lambda A_+(U^{n}_j-U^n_{j-1}) + \lambda A_- (U^n_{j+1}-u^{n}_j) = \frac{\tau}{\epsilon}
    \begin{pmatrix}
        0\\[1mm]
        f(u^{n+1}_j)-v^{n+1}_j
    \end{pmatrix},\label{original-upwind}\\[1mm]
    &L^A_-(U^{n+1}_0-U^n_0) + \lambda L^A_-A(U^n_{1}-U^{n}_0) = \frac{\tau}{\epsilon}L^A_-
    \begin{pmatrix}
        0\\[1mm]
        f(u^{n+1}_0)-v^{n+1}_0
    \end{pmatrix}, \label{original-upwindBC1}\\
    &B_uu^{n+1}_0+B_vv^{n+1}_0 = b^{n+1}. \label{original-upwindBC2}
\end{align} 
Here $\lambda=\tau/h$.
As $\epsilon$ goes to zero, the scheme \eqref{original-upwind} gives $v^{n+1}_j=f(u^{n+1}_j)$ and 
$$
u^{n+1}_j-u^n_j + \frac{\lambda}{2} \Big(f(u^{n}_{j+1})-f(u^{n}_{j-1})\Big) - \frac{\lambda}{2} \big(u^{n}_{j+1}-2u^{n}_{j}+u^{n}_{j-1}\big)=0.
$$
For $j=0$, we have $v^{n+1}_0=f(u^{n+1}_0)$ and 
$$
B_uu^{n+1}_0+B_vf(u^{n+1}_0) = b^{n+1}.
$$
However, for the case $f'(u)<0$, we know from the asymptotic analysis that the correct value of $(u,v)(0,t)$ at the boundary point does not satisfy the relation
$v(0,t)=f(u(0,t))$ due to the existence of boundary layer. This leads to an $O(1)$ error for the classical IMEX scheme at the boundary point.

To be more specific, we consider the linear case with $f=-0.5u$. The initial data are given by $u(x,0) = 2\sin(x)$ and $v(x,0) = -\sin(x).$
The boundary condition is given by 
$u(0,t)+v(0,t) = \sin(t/2)+\sin(t).$
The asymptotic solution reads as 
$$
u(x,t) = 2\sin(x+t/2)+\sin(t)\exp(-2x/\epsilon),\quad v(x,t) = -\sin(x+t/2).
$$
For sufficiently small $\epsilon$, this asymptotic solution is a good approximation to the exact solution. Based on this, we know that the correct boundary value should approximate
$$
u(0,t)=2\sin(t/2)+\sin(t),\qquad
v(0,t)=-\sin(t/2).
$$ 
But the numerical scheme gives
$$
u_0^{n+1} = 2\sin(t^{n+1}/2)+2\sin(t^{n+1}),\qquad
v_0^{n+1} = -\sin(t^{n+1}/2)-\sin(t^{n+1}).
$$
\begin{remark}
For the Jin-Xin model, the error at the boundary point occurs when $f'(u)<0$. Since the characteristic speed is negative, the error only affects the grid points near the boundary. For general relaxation model, the error could also spread to the whole domain, e.g. see Example 3 in Section \ref{Section5}. 
\end{remark}


\subsection{Boundary Asymptotic-preserving treatment}


In the upwind scheme, the spatial derivative term is discretized according to the characteristic decomposition of $A$, while no information on the equilibrium system (the sign of $f'(u)$) is involved. In this section, we consider a new spatial discretization by including the source term.
Motivated by the theoretical result for IBVPs in \cite{MR1722195}. we 
consider the matrix:
\begin{equation*}
M(\eta) = A^{-1}(I-\eta Q),\qquad Q = \begin{pmatrix}
    0 & 0\\[1mm]
    f'(u) & -1
\end{pmatrix}. 
\end{equation*}
Roughly speaking, we expect: (1) for the non-stiff case $\epsilon\gg \tau$, the parameter $\eta$ is sufficiently small and the eigenvectors of $M(\eta)$ approximate those of $A$; (2) for the stiff case $\epsilon\ll \tau$, the parameter $\eta$ is sufficiently large and the eigenvectors of $M(\eta)$ give a projection to the equilibrium system.

To this end, we first take $\eta=(\tau/\epsilon)^p$ with $p\geq 2$ a positive integer.
For the simplicity of notation, we denote $a=f'(u)$ and compute
\begin{equation}\label{M}
M(\eta) = \begin{pmatrix}
    0 & 1 \\
    1 & 0 
\end{pmatrix}
\begin{pmatrix}
    1 & 0 \\
    -\eta a & 1+\eta 
\end{pmatrix}=
\begin{pmatrix}
    -\eta a & 1+\eta \\
     1 & 0
\end{pmatrix}.
\end{equation}
One can check that $M(\eta)$ has one positive and one negative eigenvalue (see Appendix \ref{AppendA}).
Let $L_+^M$ and $L_-^M$ be the corresponding left-eigenvectors of $M$. We have
\begin{lemma}\label{lem2.1}
For sufficiently small $\eta$, it follows that
$$
L_-^M(\eta) = L^A_- + O(\eta),\quad 
L_+^M(\eta) = L^A_+ + O(\eta),
$$
where $L^A_+$ and $L^A_-$ are left-eigenvectors of $A$ associated with positive and negative eigenvalues. On the other hand, for sufficiently large $\eta$, we have 
$$
L_-^M = \left\{\begin{array}{ll}
(0,~1)+O(\eta^{-1}) &  a<0\\[3mm]
(-a,~1)+O(\eta^{-1}) & 
a>0
\end{array}\right.
~~
L_+^M = \left\{\begin{array}{ll}
   (-a,~1)+O(\eta^{-1}) &  a<0\\[3mm]
(0,~1)+O(\eta^{-1}) & 
a>0.
\end{array}\right.
$$
\end{lemma}
The proof of this lemma is based on a direct computation and we leave it to Appendix \ref{AppendA}. 
We denote 
$
(R^M_+, R^M_-) := \begin{pmatrix}
   L^M_+\\[1mm]
   L^M_-
\end{pmatrix}^{-1} 
$
and propose the following scheme.\\[2mm]
\textbf{First-order Boundary AP scheme}
\begin{align}
&U_j^{n+1}-U^n_j + \lambda AR^M_+L^M_+ (U^{n}_j-U^n_{j-1})  \nonumber\\
&\qquad \qquad \quad + \lambda AR^M_-L^M_-(U^n_{j+1}-U^{n}_j) =\frac{\tau}{\epsilon}
    \begin{pmatrix}
        0\\[1mm]
        f(u^{n+1}_j)-v^{n+1}_j
    \end{pmatrix}, \label{JX-newscheme0}\\
    &L^M_-A^{-1}(U^{n+1}_0-U^n_0) + \lambda L^M_-(U^n_{1}-U^{n}_0) = \frac{\tau}{\epsilon}L^M_-A^{-1}
    \begin{pmatrix}
        0\\[1mm]
        f(u^{n+1}_0)-v^{n+1}_0
    \end{pmatrix}, \label{JX-newschemeBC1}\\
    &B_uu^{n+1}_0+B_vv^{n+1}_0 = b^{n+1}. \label{JX-newschemeBC2}
\end{align}
Next we discuss this scheme for non-stiff case ($\epsilon\gg \tau$) and stiff case ($\epsilon\ll \tau$) separately. 
 
\textbf{Non-stiff case:}
According to our choice of $\eta$, we see that $O(\eta)=O(\tau^p)$ as $\epsilon=O(1)$. Now we show that the scheme \eqref{JX-newscheme0}---\eqref{JX-newschemeBC2} reduces to 
the classical first-order upwind scheme \eqref{original-upwind}---\eqref{original-upwindBC2} with an additional $O(\tau^p)$ error.
From Lemma \ref{lem2.1}, we compute that 
$$
AR_+^ML_+^M = AR^A_+L^A_++O(\eta) = R^A_+ \Lambda^A_+ L^A_+ +O(\eta) = A_+ +O(\eta).
$$
Similarly, we also compute $AR_-^ML_-^M= A_- +O(\eta)$. Then we see that \eqref{JX-newscheme0} reduces to the scheme \eqref{original-upwind} with an error $O(\tau^p)$.
Moreover, we compute 
$$
L^M_-A^{-1} = L^A_-A^{-1} + O(\eta) = (\Lambda_-^A)^{-1}L^A_- +O(\eta).
$$
Then \eqref{JX-newschemeBC1} becomes  
\begin{align*}
    &(\Lambda_-^A)^{-1}L^A_-(U^{n+1}_0-U^n_0) + \lambda L^A_-(U^n_{1}-U^{n}_0) = \frac{\tau}{\epsilon}(\Lambda_-^A)^{-1}L^A_-
    \begin{pmatrix}
        0\\[1mm]
        f(u^{n+1}_0)-v^{n+1}_0
    \end{pmatrix}
    +O(\tau^p).
\end{align*}
Multiplying $\Lambda^A_-$ on both sides and using the relation $\Lambda^A_-L^A_-=L^A_-A$ yield the scheme \eqref{original-upwindBC1}. At last, the scheme \eqref{JX-newschemeBC2} is the same as
\eqref{original-upwindBC2}.

\textbf{Stiff case:}
Next we discuss the stiff case where $\epsilon\ll \tau$. 
When $a<0$, we compute from Lemma \ref{lem2.1} that  
$R_-^M=(1/a,~1)^T+O(\eta^{-1})$, $R_+^M = (-1/a,~0)^T+O(\eta^{-1})$ and 
$$
AR_+^ML_+^M = \begin{pmatrix}
0 & 0\\[1mm]
1 & -1/a
\end{pmatrix}+O(\eta^{-1}),\qquad 
AR_-^ML_-^M=
\begin{pmatrix}
0 & 1\\[1mm]
0 & 1/a
\end{pmatrix}+O(\eta^{-1}).
$$
Then a direct computation of the scheme \eqref{JX-newscheme0} gives
\begin{align} 
&u^{n+1}_j - u^n_j + \lambda(v^n_{j+1}-v^{n}_j) = O\left(\dfrac{\epsilon^p}{\tau^p}\right),\qquad j\geq 1,\label{JX-BAP1}\\ 
&f(u_j^{n+1}) - v_j^{n+1} = O\left(\dfrac{\epsilon}{\tau}\right),\qquad j\geq 1.\label{JX-BAP2}
\end{align}
Moreover, we compute $L_-^MA^{-1}=(1,~0)+O(\eta^{-1})$ and \eqref{JX-newschemeBC1}---\eqref{JX-newschemeBC2} become
\begin{align}
&u^{n+1}_0 - u^n_0 + \lambda(v^n_{1}-v^{n}_0)=O\left(\dfrac{\epsilon^{p-1}}{\tau^{p-1}}\right),\label{JX-BAP3}\\
&B_uu^{n+1}_0+B_vv^{n+1}_0 = b^{n+1}.\label{JX-BAP4}
\end{align}
As $\epsilon$ goes to zero, we obtain the limiting scheme by dropping the residual terms in \eqref{JX-BAP1}-\eqref{JX-BAP3}. At first, we substitute the asymptotic limit \eqref{JX-Alimit} into the limiting scheme \eqref{JX-BAP1}. Recall Section \ref{sec:2.1} that $\bar{u}$ satisfies the equilibrium system. Thus we obtain from Taylor's expansion that
\begin{align*}
&\bar{u}^{n+1}_j - \bar{u}^n_j + \lambda(f(\bar{u}^n_{j+1})-f(\bar{u}^n_j)) = O(\tau^2),\qquad j\geq 1.
\end{align*}
Secondly, we see that the limit \eqref{JX-Alimit} satisfies \eqref{JX-BAP2} as $\epsilon\rightarrow 0$. For the boundary point $j=0$, we substitute the limit $(\bar{u}_0^{n}+\mu_0^n,~f(\bar{u}_0^{n}))$ into \eqref{JX-BAP3} and compute
\begin{align*} 
&\bar{u}_0^{n+1} - \bar{u}_0^{n} +  \lambda (f(\bar{u}_1^{n})-f(\bar{u}_0^{n})) + \mu^{n+1}_0-\mu_0^{n}
=O(\tau^2)+O(\tau).
\end{align*}
At last, we recall \eqref{JX-mu} to see that the boundary condition \eqref{JX-BAP4} is satisfied. The above discussion implies that the asymptotic limit \eqref{JX-Alimit} satisfies the limiting scheme up to  $O(\tau^2)$ for $j\geq1$ and up to $O(\tau)$ for $j=0$. To conclude, the limiting scheme of \eqref{JX-newscheme0}---\eqref{JX-newschemeBC2} is a first-order scheme for solving the asymptotic limit \eqref{JX-Alimit}.

As $a>0$, we have 
$R_-^M=(-1/a,~0)^T+O(\eta^{-1})$, $R_+^M = (1/a,~1)^T+O(\eta^{-1})$ and 
$$
AR_+^ML_+^M = \begin{pmatrix}
0 & 1\\[1mm]
0 & 1/a
\end{pmatrix}+O(\eta^{-1}),\quad 
AR_-^ML_-^M=
\begin{pmatrix}
0 & 0\\[1mm]
1 & -1/a
\end{pmatrix}+O(\eta^{-1}).
$$
Then the scheme \eqref{JX-newscheme0} gives  
\begin{align*} 
&u^{n+1}_j - u^n_j + \lambda(v^n_j-v^{n}_{j-1})=O\left(\dfrac{\epsilon^p}{\tau^p}\right),\qquad j\geq 1,\\
&f(u_j^{n+1}) - v_j^{n+1} = O\left(\dfrac{\epsilon}{\tau}\right),\qquad j\geq 1.
\end{align*} 
Moreover, we compute $L_-^MA^{-1}=(1,-a)+O(\eta^{-1})$ and \eqref{JX-newschemeBC1}---\eqref{JX-newschemeBC2} become
\begin{align*}
&f(u_0^{n+1}) - v_0^{n+1} = O\left(\dfrac{\epsilon}{\tau}\right),\\[1mm]
&B_uu^{n+1}_0+B_vv^{n+1}_0 = b^{n+1}. 
\end{align*}
In this case, there is no boundary-layer and it is easy to see that the above scheme is first-order in solving the limit $(\bar{u},f(\bar{u}))$.

\begin{remark}
In the above derivation, we use $a=f'(u)$ in the definition of the matrix $M(\eta)$. It means that 
we need to compute $f'(u_j^n)$ at each grid point. For the Jin-Xin model, there is a simpler way to implement. Namely, we use 
$a = \text{sign}(f'(u))$ in the definition of $M(\eta)$. Then, we can repeat the derivation in this subsection and verify the AP property.
In particular, the limiting scheme as 
$\epsilon\rightarrow 0$ remains the same.
\end{remark}

\section{General relaxation system}\label{Section3}
\subsection{Relaxation limits for IBVPs of relaxation systems}\label{sec:3.1}
We consider the general one dimensional linear relaxation 
system 
\begin{align}
    &U_t+AU_x=\frac{1}{\epsilon}QU,\label{eq:constantproblem}\\[1mm]
    &BU(0,t)=b(t).\label{boundaryconditions}
\end{align}
Here $U\in\mathbb{R}^{n}$ is the unknown, $A$, $Q$ are $n\times n$ matrices with constant coefficients,  and $B$ is an $n_+\times n$ constant matrix with $n_+$ the number of positive eigenvalues for $A$. 
Without loss of generality \cite{MR1693210}, we may assume the block-matrix form:
\begin{align}\label{partition}
    U=\begin{pmatrix}
        u \\
        v
    \end{pmatrix},\quad 
    A=
    \begin{pmatrix}
        A_{11} & A_{12} \\
        A_{21} & A_{22}
    \end{pmatrix},\quad 
    Q=
    \begin{pmatrix}
        0 & 0 \\
        0 & S
    \end{pmatrix},\quad 
    B=(B_u,B_v).
\end{align}
Here $u\in\mathbb{R}^{n-r}$ and $v\in\mathbb{R}^{r}$ represent the equilibrium and non-equilibrium unknowns, $S$ is an $r\times r$ negative-definite matrix.

In \cite{MR1722195}, a so-called generalized Kreiss condition (GKC) was proposed to guarantee the zero relaxation limit for the IBVPs of \eqref{eq:constantproblem}-\eqref{boundaryconditions}:

\begin{definition}[{Generalized Kreiss condition} \cite{MR1722195}] \label{def:GKC}
The boundary matrix $B$ in \eqref{boundaryconditions} satisfies the GKC, if there exist a constant $c_K>0$, such that
$$
|\det\{BR^M_+(\xi,\eta)\}|\geq c_K\sqrt{\det\{R^{M*}_+(\xi,\eta)R^{M}_+(\xi,\eta)\}},
$$
for all real number $\eta$ with $\eta\geq 0$ and complex number $\xi$ with $\textrm{Re} \xi>0$.
Here $R^M_+(\xi,\eta)$ is the right-unstable matrix for $M(\xi,\eta)=A^{-1}(\xi I-\eta Q)$.
\end{definition}
\noindent Here we use the assumption that $A$ is invertible. Note that the right-unstable matrix is defined by
\begin{definition}[Right-stable/unstable matrix \cite{MR1722195}]\label{RS-Umatrix}
Let the matrix $A\in\mathbb{R}^{n\times n}$ have precisely $k\ (0\leq k\leq n)$ stable
eigenvalues (eigenvalues with negative real parts). The full-rank matrix $R^A_- \in \mathbb{R}^{n\times k}$
is called a right-stable matrix of $A$ if
$$AR^A_-=R^A_-S$$
where $S$ is a $k\times k$ invertible matrix with $k$ stable eigenvalues. The right-unstable matrix $R^A_+$ can be defined similarly.
\end{definition}

According to the asymptotic analysis in \cite{MR1722195}, we consider the ansatz 
\begin{equation}\label{R-asy}
\begin{pmatrix}
    u_{\epsilon}\\ 
    v_{\epsilon}
\end{pmatrix}(x,t)=
\begin{pmatrix}
    \bar{u}\\ 
    \bar{v}
\end{pmatrix}(x,t)+
\begin{pmatrix}
    \mu\\ 
    \nu
\end{pmatrix}(y,t),\qquad y=\frac{x}{\epsilon}.
\end{equation}
Here $(\bar{u},\bar{v})$ is the outer solution satisfying 
$\bar{v} = 0$ and the equilibrium system
\begin{align}
\bar{u}_t + A_{11}\bar{u}_x = 0.\label{R-Eq1}
\end{align}
And $(\mu,\nu)$ represents the boundary-layer correction satisfying $(\mu,\nu)(\infty,t)\equiv0$. 
To see the equation for the boundary-layer correction term, we denote the operator 
$$
\mathcal{L}(u,v) = \begin{pmatrix}
        u \\
        v
    \end{pmatrix}_t + 
    \begin{pmatrix}
        A_{11} & A_{12} \\
        A_{21} & A_{22}
    \end{pmatrix}
    \begin{pmatrix}
        u \\
        v
    \end{pmatrix}_x -
    \frac{1}{\epsilon}
    \begin{pmatrix}
        0 & 0 \\
        0 & S
    \end{pmatrix}
    \begin{pmatrix}
        u \\
        v
    \end{pmatrix}.
$$
Note that $(u_\epsilon,v_\epsilon)$ and $(\bar{u},\bar{v})$ satisfy the equation approximately. We compute
$$
\mathcal{L}(u_{\epsilon},v_{\epsilon})-\mathcal{L}(\bar{u},\bar{v}) = \begin{pmatrix}
        \mu \\
        \nu
    \end{pmatrix}_t + \frac{1}{\epsilon}
    \begin{pmatrix}
        A_{11} & A_{12} \\
        A_{21} & A_{22}
    \end{pmatrix}
    \begin{pmatrix}
        \mu \\
        \nu
    \end{pmatrix}_y -
    \frac{1}{\epsilon}
    \begin{pmatrix}
        0 & 0 \\
        0 & S
    \end{pmatrix}
    \begin{pmatrix}
        \mu \\
        \nu
    \end{pmatrix}.
$$
Taking the coefficient of $O(1/\epsilon)$ to be zero, we have 
\begin{align}
    \mu &= -A_{11}^{-1}A_{12}\nu,\label{R-BL1}\\[1mm]
    \nu_y &= (A_{22}-A_{21}A_{11}^{-1}A_{12})^{-1}S\nu :=H\nu .\label{R-BL2}
\end{align}
Here we use the assumption that $A_{11}$ is invertible. In order to obtain the bounded solution for \eqref{R-BL2}, $\nu(0,t)$ should be given on the stable subspace. Namely, we have 
\begin{equation}\label{DeriveBC2}
L^H_+\nu(0,t)=0
\end{equation}
with $L^H_+$ the left-unstable matrix for $H$.
Now we consider the boundary condition \eqref{boundaryconditions}. Since the asymptotic solution $(u_{\epsilon},v_{\epsilon})$ satisfies the boundary condition approximately, we use the relation $\bar{v}=0$ and \eqref{R-BL1} to get
\begin{equation}\label{DeriveBC1}
B_u\bar{u}(0,t) + (B_v - B_uA_{11}^{-1}A_{12})\nu(0,t) = b(t).
\end{equation}
Moreover, we consider the decomposition $\bar{u}=R^1_+\alpha_++R^1_-\alpha_-$ where $R^1_+$ and $R^1_-$ are the right-unstable and right-stable matrices for $A_{11}$, $\alpha_+$ and $\alpha_-$ represent the incoming and outgoing modes for \eqref{R-Eq1}. By the classical theory of hyperbolic systems \cite{MR2284507}, proper boundary conditions should be prescribed for $\alpha_+$. Then \eqref{DeriveBC2}-\eqref{DeriveBC1} become 
\begin{equation}\label{DeriveBC}
\begin{pmatrix}
B_u R^1_+ &  B_v - B_uA_{11}^{-1}A_{12}\\[1mm]
0 & L^H_+
\end{pmatrix}
\begin{pmatrix}
\alpha_+\\[1mm]
\nu(0,t)
\end{pmatrix}=
\begin{pmatrix}
b(t)-B_u R^1_-\alpha_- \\[1mm]
0
\end{pmatrix}.
\end{equation}
It was proved in \cite{MR1722195} that, if the boundary condition \eqref{boundaryconditions} satisfies the GKC, then \eqref{DeriveBC} is uniquely solvable. In this way, well-posed reduced boundary conditions for \eqref{R-Eq1}
and  initial conditions for \eqref{R-BL2} can be derived.

The error estimate in \cite{MR4213673} indicates that the asymptotic solution approximates the exact solution of \eqref{eq:constantproblem}-\eqref{boundaryconditions} as $\epsilon\rightarrow 0$. On the other hand, we take $\epsilon\rightarrow 0$ in \eqref{R-asy} to obtain the limit
\begin{equation*}
U_{lim}=\begin{pmatrix}
    u_{lim}\\
    v_{lim}
\end{pmatrix}(x,t)=
\begin{pmatrix}
    \bar{u}\\
    0
\end{pmatrix}(x,t)+
\begin{pmatrix}
    -A_{11}^{-1}A_{12}\nu\\
    \nu
\end{pmatrix}(0,t).
\end{equation*}
Here $\bar{u}$ is solved from the equilibrium system \eqref{R-Eq1} with reduced boundary condition and $\nu(0,t)$ is solved from \eqref{DeriveBC}. To express the asymptotic limit on gird points $x_j=jh$, $t_n=n\tau$ with $j,n\geq 0$, we denote
\begin{equation}\label{R-limit}
(u_{lim},v_{lim})(x_j,t_n)=\left\{\begin{array}{lc}
(\bar{u}_0^n-A_{11}^{-1}A_{12}\nu_0^n,~\nu_0^n)\qquad  & j=0, \\[2mm]
(\bar{u}_j^n,~0)  \qquad   & j\geq 1,
\end{array}\right.
\end{equation}
where $\bar{u}_j^n=\bar{u}(x_j,t_n)$ and $\nu_0^n=\nu(0,t_n)$.

\subsection{Boundary AP treatment}\label{section3.2}

Recall the boundary AP scheme for the Jin-Xin model. It seems natural to compute the eigenvectors of $M(1,\eta)=A^{-1}(I-\eta Q)$
and use the scheme as that in \eqref{JX-newscheme0}---\eqref{JX-newschemeBC2}. However, for general $A$ and $Q$, the eigenvectors are not necessarily analytical with respect to the parameter $\eta$ according to the perturbation theory of matrices \cite{MR678094}. 
In this case, we can not show the AP property as that in the previous section. 

Therefore, for general relaxation systems, we only consider two limiting case, i.e. $\eta=0$ and $\eta=\infty$.
Denote the right-stable and right-unstable matrix for $M(1,\eta)$ by $R_-^M(\eta)$ and $R_+^M(\eta)$.
When $\eta=0$, we know that $M(1,0)=A^{-1}$ and
$$
R_+^M(0) = R_+^A,\quad R_-^M(0) = R_-^A,
$$
where $R_+^A$ and $R_-^A$ are right-unstable and right-stable matrices of $A$.
On the other hand, it was shown in \cite{MR1722195} that
\begin{lemma}\label{lem3.1}
    If the coefficient matrix $A$ and $A_{11}$ are both invertible, the right-unstable and stable matrices $R_+^M(\eta)$ and $R_-^M(\eta)$ for $M(1,\eta)$ have the following form as $\eta\rightarrow \infty$: 
    $$
R_{\pm}^\infty:=\lim_{\eta\rightarrow\infty}R_{\pm}^M(\eta) = \begin{pmatrix}
        I_{n-r} & -A_{11}^{-1}A_{12} \\[1mm]
        0 & I_r
\end{pmatrix}
\begin{pmatrix}
    R_{\pm}^1 & 0 \\[1mm]
    0 & R_{\mp}^H
\end{pmatrix},
    $$
    where $R_+^1$ and $R_+^H$ ($R_-^1$ and $R_-^H$) are right-unstable (stable) matrices for $A_{11}$ and $H$.
\end{lemma}
Having this, we consider the following matrix 
\begin{equation}\label{RMS}
R_{\pm}=R_{\pm}(\tau,\epsilon)=\left\{\begin{array}{lc}
R_{\pm}^A & \quad\tau<\epsilon, \\[2mm]
R_{\pm}^\infty & \quad \tau\geq\epsilon.
\end{array}\right.
\end{equation}
Denote $L_+$, $L_-$ such that
$\begin{pmatrix}
L_+\\
L_-
\end{pmatrix}
(R_+,~R_-) = I$.
We propose 
\begin{eqnarray}\label{newscheme0}
    \left\{\begin{array}{ll}
    U^{n+1}_j - U^n_j + \lambda AR_+L_+ (U^{n}_j-U^n_{j-1}) 
    \\[3mm]
    \qquad \qquad\quad + \lambda AR_-L_-(U^n_{j+1}-U^{n}_j)=\dfrac{\tau}{\epsilon}QU^{n+1}_j,\quad j\geq 1\\[3mm]
    L_-A^{-1}(U^{n+1}_0 -U^n_0) + \lambda   L_- (U^n_{1}-U^{n}_0)=\dfrac{\tau}{\epsilon}L_-A^{-1}Q U^{n+1}_0,\\[3mm]
    BU^{n+1}_0 = b^{n+1}.
    \end{array}\right.
\end{eqnarray}

Clearly, when $\tau\ll\epsilon$, our scheme is just the original first-order upwind scheme since $AR_{\pm}L_{\pm}=AR^A_{\pm}L^A_{\pm}=A_{\pm}$.
Next we discuss the scheme for the stiff case where $\tau \gg \epsilon$. As a preparation, we first recall the notation in Lemma \ref{lem3.1} and define $L_{\pm}^1$, $L_{\pm}^X$ such that
$$
\begin{pmatrix}
L_+^1\\[1mm]
L_-^1
\end{pmatrix}
\begin{pmatrix}
R_+^1 &
R_-^1
\end{pmatrix} = I,\quad 
\begin{pmatrix}
L_+^X\\[1mm]
L_-^X
\end{pmatrix}
\begin{pmatrix}
R_+^X &
R_-^X
\end{pmatrix} = I.
$$
From Lemma \ref{lem3.1}, we compute
\begin{align*}
L_{\pm}
 = 
\begin{pmatrix}
    L_{\pm}^1 & 0\\[2mm]
    0 & L_{\mp}^H
\end{pmatrix}
\begin{pmatrix}
        I_{n-r} & A_{11}^{-1}A_{12} \\[2mm]
        0 & I_r
\end{pmatrix}
\end{align*}
for the case $\epsilon\ll\tau$.
A direct computation shows that
\begin{align*}
AR_{\pm}L_{\pm}
=&\begin{pmatrix}
    A_{11}R_{\pm}^1L_{\pm}^1 & A_{11}R_{\pm}^1L_{\pm}^1A_{11}^{-1}A_{12} \\[2mm]
    \star & \star
\end{pmatrix}.
\end{align*}
According to the definition of right-unstable and stable matrices, we have 
$$
A_{11}R_+^1L_+^1 = R_+^1\Lambda_+^1L_+^1:=A_{11}^+,\qquad
A_{11}R_-^1L_-^1 = R_-^1\Lambda_-^1L_-^1:=A_{11}^-.
$$
Here $\Lambda_+^1$ and $\Lambda_-^1$ are positive and negative eigenvalues of $A_{11}$. By this block-matrix form, we rewrite the first row in  \eqref{newscheme0} and take $\epsilon\rightarrow 0$ to obtain
\begin{align}
&u^{n+1}_1 - u^n_1 + \lambda A_{11}^+(u^{n}_1-u^n_{0}-A_{11}^{-1}A_{12}v_0^n) + \lambda A_{11}^-
(u^n_{2}-u^{n}_1) =0,\quad j=1\label{R-limscheme1}\\[2mm]
&u^{n+1}_j - u^n_j + \lambda A_{11}^+(u^{n}_j-u^n_{j-1})+ \lambda A_{11}^-
(u^n_{j+1}-u^{n}_j) =0,\qquad j\geq 2\label{R-limscheme2}\\[2mm]
&v_{j}^{n+1} = 0,\quad j\geq 1.\label{R-limscheme3}
\end{align}
For the boundary point $j=0$, we compute
$$
A^{-1} = 
\begin{pmatrix}
        I_{n-r} & -A_{11}^{-1}A_{12} \\[2mm]
        0 & I_r
\end{pmatrix}
\begin{pmatrix}
    A_{11}^{-1} & 0 \\[2mm]
    Y & X
\end{pmatrix},
\quad 
L_-A^{-1} = \begin{pmatrix}
    L_{-}^1 & 0\\[2mm]
    0 & L_{+}^H
\end{pmatrix}\begin{pmatrix}
    A_{11}^{-1} & 0 \\[2mm]
    Y & X
\end{pmatrix}
$$
with $X=(A_{22}-A_{21}A_{11}^{-1}A_{12})^{-1}$
and $Y=-XA_{21}A_{11}^{-1}$. 
Then the scheme \eqref{newscheme0} gives
\begin{equation*}
\begin{pmatrix}
    L_{-}^1A_{11}^{-1} & 0 \\[2mm]
    L_{+}^H Y & L_{+}^H X
\end{pmatrix}
\left[
\begin{pmatrix}
u_0^{n+1} - u_0^n\\[2mm]
v_0^{n+1} - v_0^n    
\end{pmatrix} + \lambda 
\begin{pmatrix}
        A_{11} & A_{12} \\[2mm]
        A_{21} & A_{22}
    \end{pmatrix}
\begin{pmatrix}
u_1^n - u_0^n\\[2mm]
v_1^n - v_0^n    
\end{pmatrix} 
- \frac{\tau}{\epsilon} 
\begin{pmatrix}
0 \\[2mm]
S v_0^{n+1} 
\end{pmatrix}
\right] = 0.
\end{equation*}
Taking limit $\epsilon \rightarrow 0$, we obtain
\begin{align*} 
& L_{-}^1 A_{11}^{-1} (u_0^{n+1} - u_0^n) 
+ \lambda L_{-}^1 \Big[(u_1^n - u_0^n) 
+ A_{11}^{-1}A_{12} (v_1^n - v_0^n)\Big] = 0,\\[1mm]
&L_{+}^H XS v_0^{n+1} = 0. 
\end{align*}
Recall that $XS=H$ and $L_+^H$ is the left-unstable matrix for $H$. We have
\begin{align}
&\qquad L_{+}^H v_0^{n+1} = 0.\label{R-limscheme41} 
\end{align}
Besides, since $L_{-}^1 A_{11}^{-1} = (\Lambda_{-}^1)^{-1} L_{-}^1$, we obtain 
\begin{align}
& L_{-}^1 (u_0^{n+1} - u_0^n) 
+ \lambda L_{-}^1 \Big[A_{11}(u_1^n - u_0^n) 
+ A_{12} (v_1^n - v_0^n)\Big] = 0.\label{R-limscheme4}
\end{align}

Now we substitute the asymptotic limit \eqref{R-limit} into the limiting scheme \eqref{R-limscheme1}-\eqref{R-limscheme4}. Recall that $\bar{u}$ satisfies the equilibrium system $\bar{u}_t+A_{11}\bar{u}_x=0$. We substitute the relaxation limit \eqref{R-limit} into \eqref{R-limscheme1} and use the Taylor expansion to compute 
$$
\bar{u}^{n+1}_1 - \bar{u}^n_1 + \lambda A_{11}^+(\bar{u}^{n}_1-\bar{u}^n_{0}) + \lambda A_{11}^-
(\bar{u}^n_{2}-\bar{u}^{n}_1) =O\left(\tau^2\right).
$$
Similarly, for \eqref{R-limscheme2} and \eqref{R-limscheme4}, we compute:
\begin{align*}
&\bar{u}^{n+1}_j - \bar{u}^n_j + \lambda A_{11}^+(\bar{u}^{n}_j-\bar{u}^n_{j-1}) + \lambda A_{11}^-
(\bar{u}^n_{j+1}-\bar{u}^{n}_j) =O\left(\tau^2\right),\quad j\geq 2\\[2mm]
&L_{-}^1 (\bar{u}_0^{n+1} - \bar{u}_0^n) 
+ \lambda L_{-}^1 A_{11}(\bar{u}_1^n - \bar{u}_0^n)= O(\tau^2)+L_{-}^1 A_{11}^{-1}A_{12}(\nu_0^{n+1} - \nu_0^n) =O(\tau).
\end{align*}
Clearly, the limit satisfies the relation \eqref{R-limscheme3}. Furthermore, we know from Section \ref{sec:3.1} that $\bar{u}_0^{n+1}$ and $\nu_0^{n+1}$ satisfy $L_+^H\nu_0^{n+1}=0$ and $B_u\bar{u}_0^{n+1}+(B_v - B_uA_{11}^{-1}A_{12})\nu_0^{n+1} = b^{n+1}$. To sum up, we conclude that the scheme \eqref{R-limscheme1}-\eqref{R-limscheme4}, together with the boundary condition in \eqref{newscheme0}, is a first-order scheme for solving 
the limit \eqref{R-limit}.

\section{Interface problem}\label{Section4}
In this section, we show that our scheme can be also used to compute the interface problem:
\begin{align*}
    &U_t+AU_x=\frac{1}{\epsilon(x)}QU,\quad x\in \mathbb{R}.
\end{align*}
Here $A$, $Q$ are constant matrices which have the same partition as that in \eqref{partition}
and 
\begin{equation}\label{Interface-ep}
\epsilon(x)=1,\quad x<0,\qquad
\epsilon(x)=\epsilon_0\ll1,\quad x\geq0.
\end{equation}
Denote $U^l(x,t)=U(-x,t)$ for $x<0$ and $U^r(x,t)=U(x,t)$ for $x>0$. According to the analysis in \cite{MR4593212}, the interface condition at $x=0$ is $U^l(0,t)=U^r(0,t)$. Thus we rewrite the interface problem as the IBVP:
\begin{eqnarray}\label{InterfaceProblem}
    \left\{\begin{array}{l}
    \begin{pmatrix}
    U^l\\
    U^r
    \end{pmatrix}_t+
    \begin{pmatrix}
    -A &  \\
       & A
    \end{pmatrix}
    \begin{pmatrix}
    U^l\\
    U^r
    \end{pmatrix}_x=\begin{pmatrix}
    Q &  \\
       & Q/\epsilon_0
    \end{pmatrix}
    \begin{pmatrix}
    U^l\\ 
    U^r
    \end{pmatrix},\quad x>0,\\[5mm]
    ~(I,-I)\begin{pmatrix}
    U^l\\ 
    U^r
    \end{pmatrix}(0,t)=0.
    \end{array}\right.
\end{eqnarray}
We consider the following asymptotic solution for $U^r$:
$$
U^r_\epsilon = 
\begin{pmatrix}
    \bar{u}\\
    \bar{v}
\end{pmatrix}(x,t)+
\begin{pmatrix}
    \mu\\
    \nu
\end{pmatrix}(x/\epsilon,t).
$$
Here $\bar{u}$ satisfies the equilibrium system 
$\bar{u}_t+A_{11}\bar{u}_x=0$ and $(\mu,\nu)$ represents the boundary-layer correction. 
Using a similar derivation as that in Section \ref{Section3}, we know that  the asymptotic solution $U^r_{\epsilon}$ converges to
$$
U^r_{lim}(x,t)=\begin{pmatrix}
    \bar{u}\\
    0
\end{pmatrix}(x,t)+
\begin{pmatrix}
    -A_{11}^{-1}A_{12}\nu\\
    \nu
\end{pmatrix}(0,t).
$$
Moreover, $U^l$ satisfies the $\epsilon$-independent equation $U^l_t-AU^l_x=QU^l$ and the coupling condition $U^l(0,t)=U^r_{lim}(0,t)$ holds.
The solvability of unknowns $U^l$, $\bar{u}$ and $\nu$ has been shown in \cite{MR4593212}. Having these, we denote the relaxation limit on gird points $(x_j,t_n)$ by $U^l(x_j,t_n)$ and 
\begin{equation}\label{limit-interface}
U^r_{lim}(x_j,t_n)=\left\{\begin{array}{lc}
(\bar{u}_0^n-A_{11}^{-1}A_{12}\nu_0^n,~\nu_0^n)\quad  & j=0, \\[2mm]
(\bar{u}_j^n,~0)  \quad   & j\geq 1,
\end{array}\right.
\end{equation}
where $\bar{u}_j^n=\bar{u}(x_j,t_n)$ and $\nu_0^n=\nu(0,t_n)$. 

Now we turn to discuss the numerical scheme for the interface problem \eqref{InterfaceProblem}. Note the block-diagonal form of coefficient matrices in \eqref{InterfaceProblem}. We recall \eqref{RMS} and define the following matrix:
$$
R_\pm=
\begin{pmatrix}
R^{M_1}_\pm  & \\ 
& R^{M_2}_\pm
\end{pmatrix}:=
\begin{pmatrix}
R^{A}_\mp & \\ 
& R_{\pm}^\infty
\end{pmatrix}.
$$
Here $R_{\pm}^\infty$ is the matrix given in Lemma \ref{lem3.1}. Notice that we take $R^{M_1}_\pm=R^{A}_\mp$ because the coefficient matrix for the equation $U^l_t-AU^l_x=QU^l$ is $-A$. 
Furthermore, we denote $L_\pm^{M_1}$ and $L_\pm^{M_2}$ such that
$$
\begin{pmatrix}
L_+^{M_1}\\[1mm]
L_-^{M_1}
\end{pmatrix}
(R_+^{M_1},~R_-^{M_1}) = I,\quad
\begin{pmatrix}
L_+^{M_2}\\[1mm]
L_-^{M_2}
\end{pmatrix}
(R_+^{M_2},~R_-^{M_2}) = I.
$$
Then we consider the following scheme for the IBVP \eqref{InterfaceProblem}:
\begin{small}
\begin{align}
& (I- \tau Q)U^{l,n+1}_j = U^{l,n}_j - \lambda \Big[-AR^{M_1}_+L^{M_1}_+(U^{l,n}_j-U^{l,n}_{j-1}) - AR^{M_1}_-L^{M_1}_-(U^{l,n}_{j+1}-U^{l,n}_{j})\Big],\label{interface-eq1}\\[1mm]
& (I- \frac{\tau}{\epsilon_0} Q)U^{r,n+1}_j
 =U^{r,n}_j - \lambda \Big[AR^{M_2}_+L^{M_2}_+(U^{r,n}_j-U^{r,n}_{j-1}) + AR^{M_2}_-L^{M_2}_-(U^{r,n}_{j+1}-U^{r,n}_{j})\Big]\label{interface-eq2}
\end{align}
\end{small}
for $j\geq 1$ and 
\begin{small}
\begin{align}
    &L^{M_1}_-(-A)^{-1}(I- \tau Q)U^{l,n+1}_0 = L^{M_1}_-(-A)^{-1}U^{l,n}_0 - \lambda   L^{M_1}_- (U^{l,n}_{1}-U^{l,n}_0),\label{interface-eq3}\\[1mm]
    &L^{M_2}_-A^{-1}(I- \frac{\tau}{\epsilon_0} Q)U^{r,n+1}_0 = L^{M_2}_-A^{-1}U^{r,n}_0 - \lambda   L^{M_2}_- (U^{r,n}_{1}-U^{r,n}_0),\label{interface-eq4}\\
    &U^{l,n+1}_0 = U^{r,n+1}_0.\label{interface-eq5}
\end{align}
\end{small}
Clearly, \eqref{interface-eq1} and
\eqref{interface-eq3} are just the first-order upwind scheme for solving the equation $U^l_t-AU^l_x=QU^l$. 
On the other hand, by using the same argument as that in Section \ref{section3.2}, it is not difficult to check that the limiting scheme of \eqref{interface-eq2} and \eqref{interface-eq4} is first-order in solving the asymptotic limit $U^r_{lim}$ in \eqref{limit-interface}.

At last, we rewrite \eqref{interface-eq1}---\eqref{interface-eq5} as the scheme for the original interface problem on $\mathbb{R}$.
Denote the grid $x_j=jh$ for $j=0,\pm 1,\pm 2,\cdots$ and take
$$
U_{j}^n= \left\{
\begin{array}{lc}
   U_{-j}^{l,n}  & j\leq0, \\[3mm]
   U_{j}^{r,n}  & j>0.
\end{array}\right.
$$
Define 
$\epsilon^r_j=\epsilon(x^+_j)$ and $\epsilon^l_j=\epsilon(x^-_j)$, which are right and left limits of $\epsilon(x)$ at $x=x_j$. Clearly, we have 
$$
\epsilon^r_j= \left\{
\begin{array}{lc}
   1  & j<0, \\[1mm]
   \epsilon_0  & j\geq0,
\end{array}\right.\qquad 
\epsilon^l_j= \left\{
\begin{array}{lc}
   1  & j\leq0, \\[1mm]
   \epsilon_0  & j>0.
\end{array}\right.
$$
Correspondingly, we also define the left stable (unstable) matrices by  
$$
L^{r,j}_{-}= \left\{
\begin{array}{lc}
   L^{M_1}_{+}  & j<0, \\[3mm]
   L^{M_2}_{-}  & j\geq0,
\end{array}\right.\qquad 
L^{l,j}_{+}= \left\{
\begin{array}{lc}
   L^{M_1}_{-}  & j\leq0, \\[3mm]
   L^{M_2}_{+}  & j>0.
\end{array}\right.
$$
Then the scheme \eqref{interface-eq1}---\eqref{interface-eq5} can be rewritten as
\begin{align}
&L^{r,j}_-A^{-1}\Big(I-\dfrac{\tau}{\epsilon_j^r}Q\Big)U^{n+1}_j
 = L^{r,j}_- A^{-1} U^{n}_j - \lambda L^{r,j}_- (U^{n}_{j+1}-U^{n}_{j}), \label{interface-newform1}\\
&L^{l,j}_+ A^{-1}\Big(I-\dfrac{\tau}{\epsilon_j^l}Q\Big)U^{n+1}_j
 = L^{l,j}_+ A^{-1} U^{n}_j - \lambda L^{l,j}_+ (U^{n}_{j}-U^{n}_{j-1}). \label{interface-newform2}
\end{align}
This form is equivalent to the scheme 
\eqref{interface-eq1}---\eqref{interface-eq5} for the interface problem with $\epsilon(x)$ given in \eqref{Interface-ep}. In application, the relaxation parameter $\epsilon(x)$ can be also given in more complicated way. In such cases, it is easier to extend our scheme by using this form.
Namely, we first compute $\epsilon^r_j=\epsilon(x^+_j)$ and $\epsilon^l_j=\epsilon(x^-_j)$ at each point $x_j$. Then we take
$$
L^{r,j}_{-}=\left\{\begin{array}{lc}
L_{-}^A & \quad \epsilon_j^r>\tau, \\[1mm]
L_{-}^\infty & \quad \epsilon_j^r\leq\tau,
\end{array}\right.\qquad 
L^{l,j}_{+}=\left\{\begin{array}{lc}
L_{+}^A & \quad\epsilon_j^l>\tau, \\[1mm]
L_{+}^\infty & \quad \epsilon_j^l\leq\tau.
\end{array}\right.
$$ 
Having these, we use  \eqref{interface-newform1}-\eqref{interface-newform2} to update $U^{n+1}_j$ for each $j=0,\pm 1,\pm 2,\cdots$.

\section{Numerical experiment}\label{Section5}
~

\textbf{Example 1 (Linear Jin-Xin model):} 
In this numerical experiment, we check the validity of the boundary AP scheme by showing the convergence rate. 
We consider the Jin-Xin model \eqref{JX-sys}-\eqref{JX-BC} with linear function $f(u)=au$. Particularly, we consider three cases.

Case 1: We first consider the case with $a=-0.5<0$. In this case, there is boundary-layer for sufficiently small $\epsilon$. In order to satisfy the assumption $\epsilon \ll \tau$, we take $\epsilon=10^{-9}$.
The initial data for \eqref{JX-sys} are given by 
$u(x,0) = 2\sin(x)$ and $v(x,0) = -\sin(x)$. 
The boundary condition \eqref{JX-BC} is given by 
$
u(0,t)+v(0,t) = \sin(t/2)+\sin(t).
$
Since $\epsilon$ is sufficiently small, we take the asymptotic solution 
$$
u_{\epsilon}(x,t) = 2\sin(x+t/2)+\sin(t)\exp(-2x/\epsilon),\quad v_{\epsilon}(x,t) = -\sin(x+t/2)
$$
as the reference solution.
The computation domain is $x\in[0,2]$. The solution is computed to $t=0.5$. We take the CFL-number as $C_{CFL}=0.8$ and the parameter $p$ in \eqref{M} as $p=2$. We compute the $L^2$, $L^1$ and $L^\infty$ errors between the numerical solutions $(u,v)$ and reference solutions $(u_{\epsilon},v_{\epsilon})$. The result shown in Figure \ref{fig:JX-convergence} (left) implies that our scheme has the first-order accuracy even for the stiff case with boundary-layer.

Case 2:
We take $a= 0.5>0$. There is no boundary-layer in this case. The initial data are given by 
$u(x,0) = 2\sin(x)$ and $v(x,0) = \sin(x).$
The boundary condition is given by 
$
u(0,t)+v(0,t) = -3\sin(t/2).
$
The asymptotic solution reads as 
$$
u_{\epsilon}(x,t) = 2\sin(x-t/2),\quad v_{\epsilon}(x,t) = \sin(x-t/2).
$$
The other computational parameters are the same as the case 1 and the numerical result is shown in Figure \ref{fig:JX-convergence} (middle). Similarly, the scheme has first-order accuracy.

\begin{figure}[h]
\centering
\includegraphics[width=5.0in]{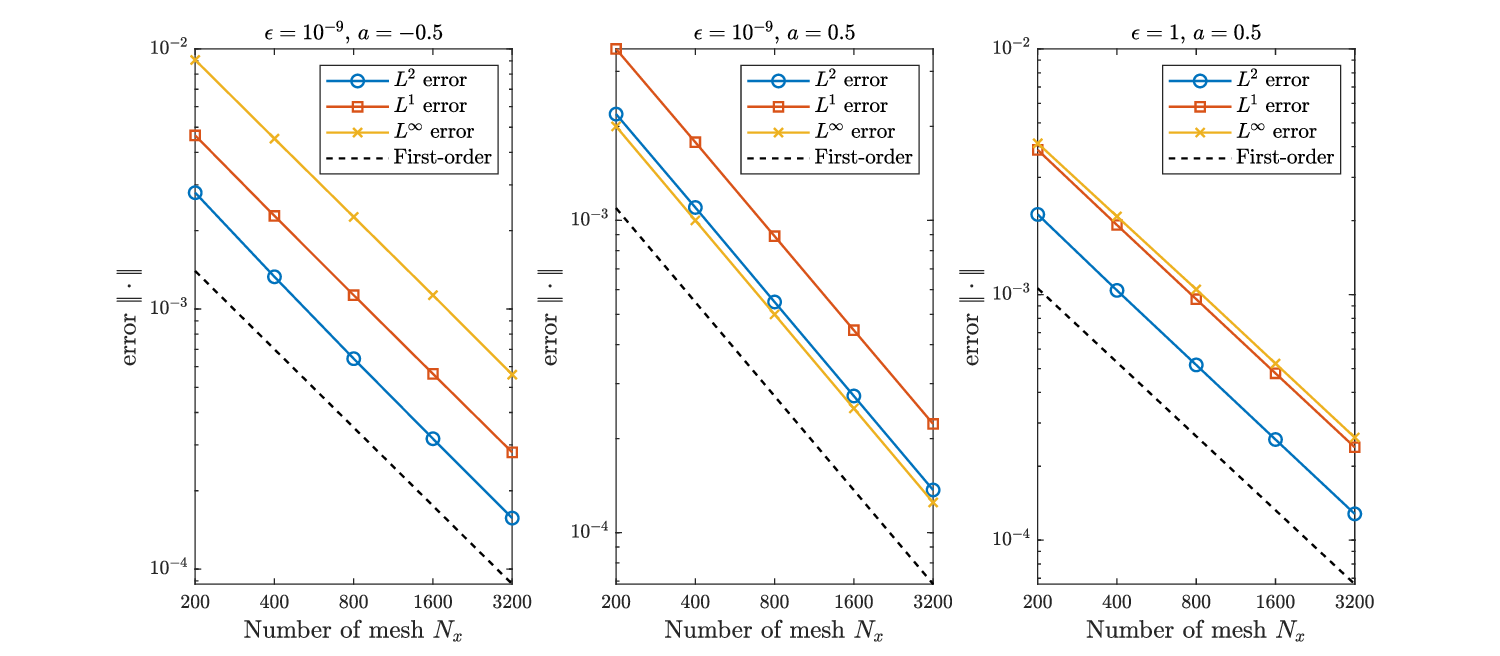}
\caption{$L^2$, $L^1$ and $L^\infty$ errors between the numerical solutions and reference solutions for different cases: (left) the stiff case with boundary-layer; (middle) the stiff case without boundary-layer; (right) the non-stiff case.} 
\label{fig:JX-convergence}
\end{figure}



Case 3:
We consider the non-stiff case with $\epsilon=1$. 
In order to easily get the analytical solution, we consider the following Jin-Xin model with force term:
\begin{align*}
&u_t + v_x = 0,\\[1mm]
&v_t + u_x = 0.5u-v +f,
\end{align*}
where $f=-1.5\sin(x+t)$.
The initial condition is given by $u(x,0) = \sin(x), ~v(x,0) = -\sin(x)$ and the boundary condition is $u(0,t)+v(0,t) = 0$. The exact solution reads 
$u_e(x,t) = \sin(x+t)$ and $v_e(x,t)=-\sin(x+t).$
We directly take this as the reference solution and show the convergence rate in Figure \ref{fig:JX-convergence} (right). It implies that the scheme is a first-order scheme for the non-stiff relaxation system with $\epsilon=1$.

\textbf{Example 2 (Nonlinear Jin-Xin model):}
In this experiment, we exploit the method in Section \ref{Section2} and show the result for nonlinear Jin-Xin model. Moreover, we compare our result with the one computed from the original upwind scheme.

We consider the Jin-Xin model \eqref{JX-sys} for the nonlinear case with $f(u)=\frac{1}{4}(e^{-u}-1)$. The initial condition is given by $u(x,0) = \sin^3(\pi x)$ and $v(x,0) = f(u(x,0))$. The boundary condition reads $u(0,t)+v(0,t) = \sin(2t)$.
We divide the computation domain $x\in[0,1]$ into $N_x=800$ cells. Besides, we take $\tau=5\times 10^{-4}$, $\epsilon=10^{-6}$ and $p=2$. The reference solution is taken as the asymptotic solution \eqref{JX-asy}, $\bar{u}$ is computed numerically from the equilibrium system, and $\mu(0,t)$ is computed from the boundary condition. The numerical result is shown for $t=0.2$ in Figure \ref{fig:1} (left). In comparison, we also plot the result by using the original upwind scheme, see Figure \ref{fig:1} (right). From the figures, we see that the upwind scheme can not capture the correct value at the boundary point. In our method, the boundary-layer is well captured even the mesh is relative coarse ($h=1.25\times 10^{-3}$ and $\epsilon=10^{-6}$).

\begin{figure}[h]
\centering
\includegraphics[width=2.5in]{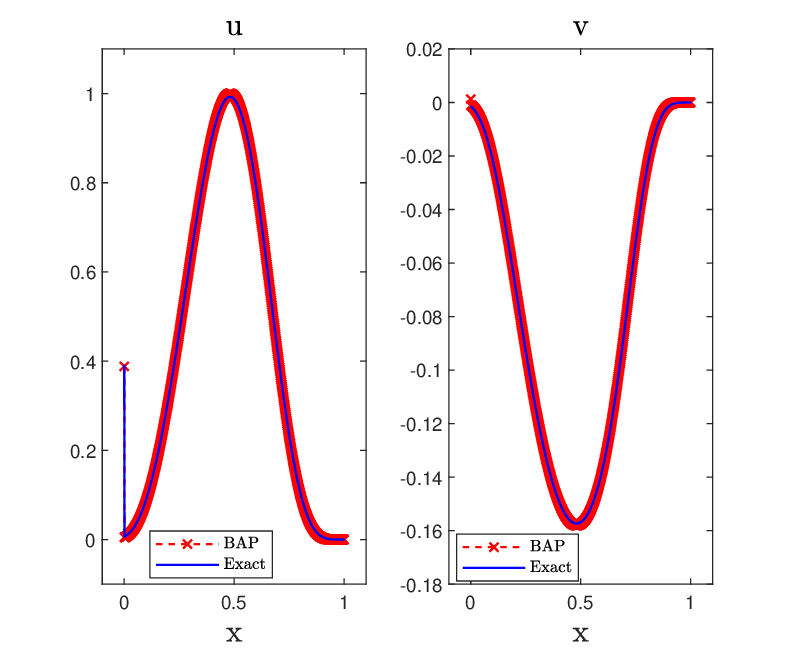}
\includegraphics[width=2.5in]{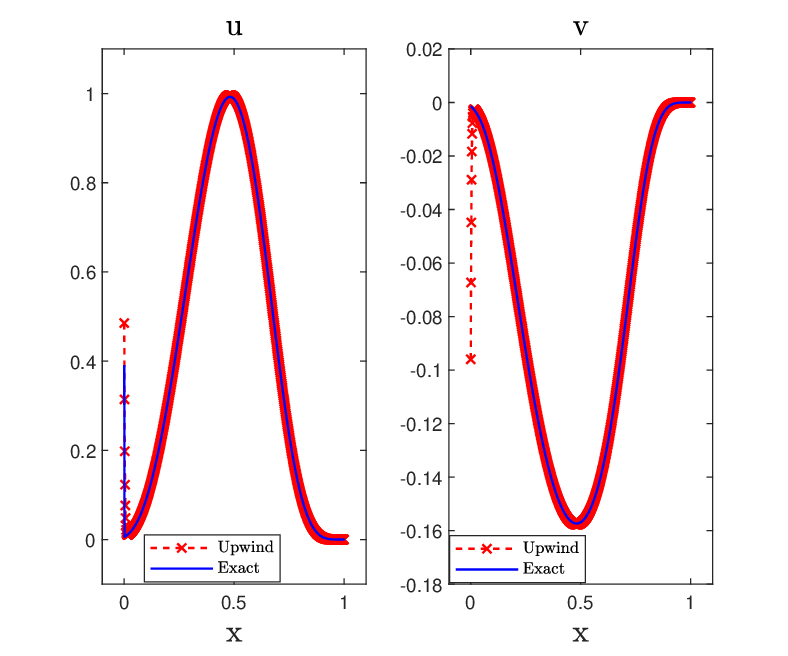}
\caption{(Left) the numerical results at $t=0.2$ with boundary-AP treatment; (Right) the numerical results at $t=0.2$ by upwind scheme.} 
\label{fig:1}
\end{figure}



\textbf{Example 3 (A linear relaxation model):}
Consider the relaxation system
\begin{align*}
&
\left\{\begin{array}{l}
\partial_t u+\partial_x v =0, \\[3mm]
\partial_t v+\partial_x (u+p) =0, \\[1mm]
\partial_t p+\partial_x (v+p) =-\dfrac{1}{\epsilon}p,
\end{array}\right.\\[2mm]
\text{boundary condition:}\quad &u(0,t)=-\sin t,\quad v(0,t)+2p(0,t)=0,\\[2mm]
\text{initial condition:}\quad&u(x,0)=2\sin x,\quad 
v(x,0)\equiv0,\quad p(x,0)\equiv0.
\end{align*}
We build this artificial example to illustrate the necessity of the boundary-AP treatment. It is not difficult to check that this model satisfies the structural stability condition in \cite{MR1693210}. Moreover, one can check that the boundary condition satisfies the strictly dissipative condition and thereby the GKC \cite{MR1693210}. These imply the existence of zero-relaxation limit for the IBVP.

For sufficiently small $\epsilon=10^{-6}$, we consider the following asymptotic solution as the reference solution 
\begin{eqnarray*}
\left\{
\begin{array}{l}
u_\epsilon=\sin(x-t)+\sin(x+t)-\sin t\exp\left(-x/\epsilon\right), \\[2mm]
v_\epsilon=\sin(x-t)-\sin(x+t),\quad 
p_\epsilon=\sin t\exp\left(-x/\epsilon\right).
\end{array}\right.
\end{eqnarray*}
The computation domain $x\in[0,0.5]$ is divided into $N_x=100$ cells. We use our method to compute the numerical solution with  $\tau=10^{-3}$. Besides, we also use the classical upwind scheme with the same parameters. The numerical result is shown for $t=0.3$ in Figure \ref{fig:3}. To exhibit the result clearly, we only plot the solution on a subset of the grid points (including boundary point $x=0$). From the figures, we see that the upwind scheme can not capture the right boundary value and the error may spread as time evolves. In our method, the boundary-layer is well captured and the numerical results match the reference solution well.
\begin{figure}[h]
\centering
\includegraphics[width=5.0in]{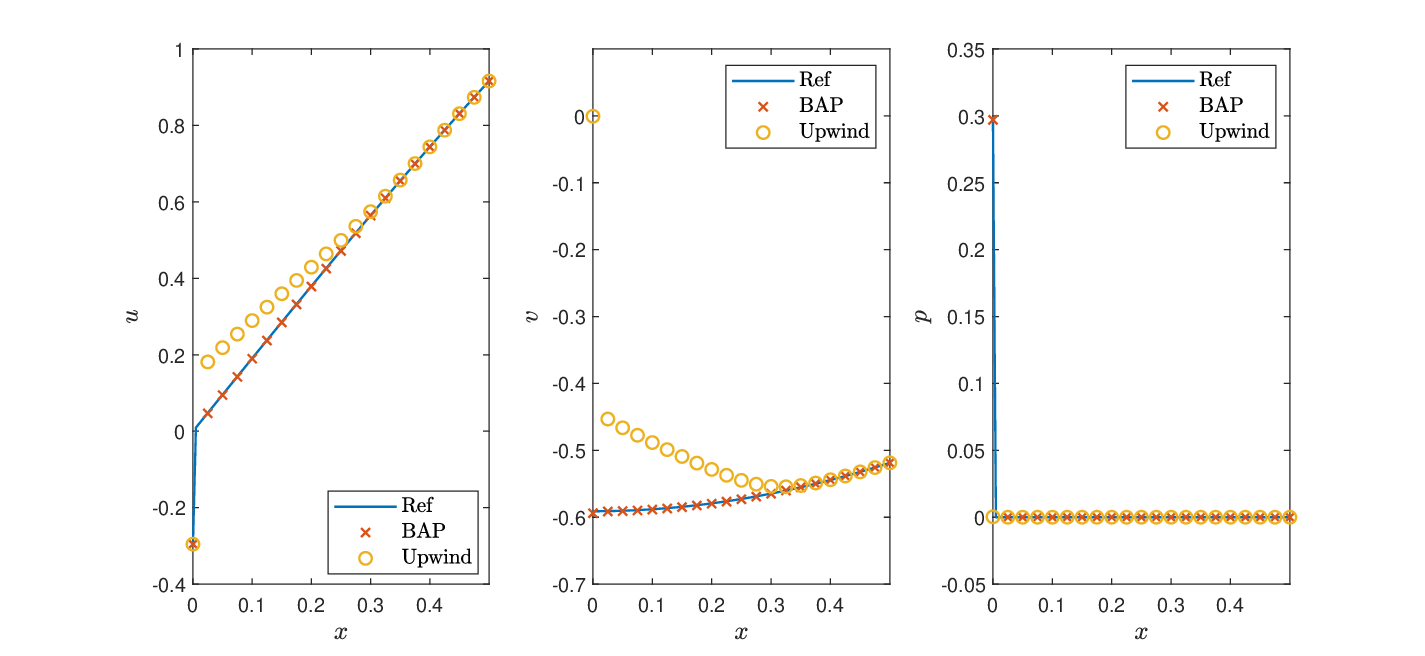}
\caption{The numerical results at $t=0.3$ by our boundary-AP method and by the classical upwind method.} 
\label{fig:3}
\end{figure}

\textbf{Example 4 (Interface problem for the Jin-Xin model):}
In this experiment, we consider the Jin-Xin model 
\begin{align*}
&\left\{\begin{array}{l}
u_t + v_x = 0,\\[2mm]
v_t + u_x = (f(u)-v)/\epsilon(x),
\end{array}
\right. \qquad x\in \mathbb{R},~t>0, 
\end{align*} with $f(u)=\frac{1}{4}(e^{-u}-1)$. The relaxation parameter $\epsilon(x)$ is given by \eqref{Interface-ep} with $\epsilon_0=10^{-4}$. The initial data are given by 
$u(x,0) = \sin(\pi x)$ and $v(x,0) = f(\sin(\pi x)).$
We divide the computation domain  $x\in[-1,1]$ into $N_x$ cells. In our method, we take $p=4$. The reference solution is computed by upwind scheme with sufficiently fine mesh $h=5\times 10^{-5}$. The numerical result is shown for $t=0.4$ in Figure \ref{fig:4}. The left and middle figures are the results computed by the original upwind scheme and by our method with $N_x=100$. From these two figures, we see that our method has better performance than the upwind scheme. For the classical upwind scheme, the error seems smaller if the mesh size $h$ is reduced (but still larger than $\epsilon$). However, we observe that the upwind scheme may lead to an incorrect spike near the interface even with a fine mesh. In Figure \ref{fig:4} (right), we plot the numerical result of $v$ by using the upwind scheme with $N_x=1000$. This phenomenon is similar to the incorrect boundary value in Example 2.

\begin{figure}[h]
\centering
\includegraphics[width=5.0in]{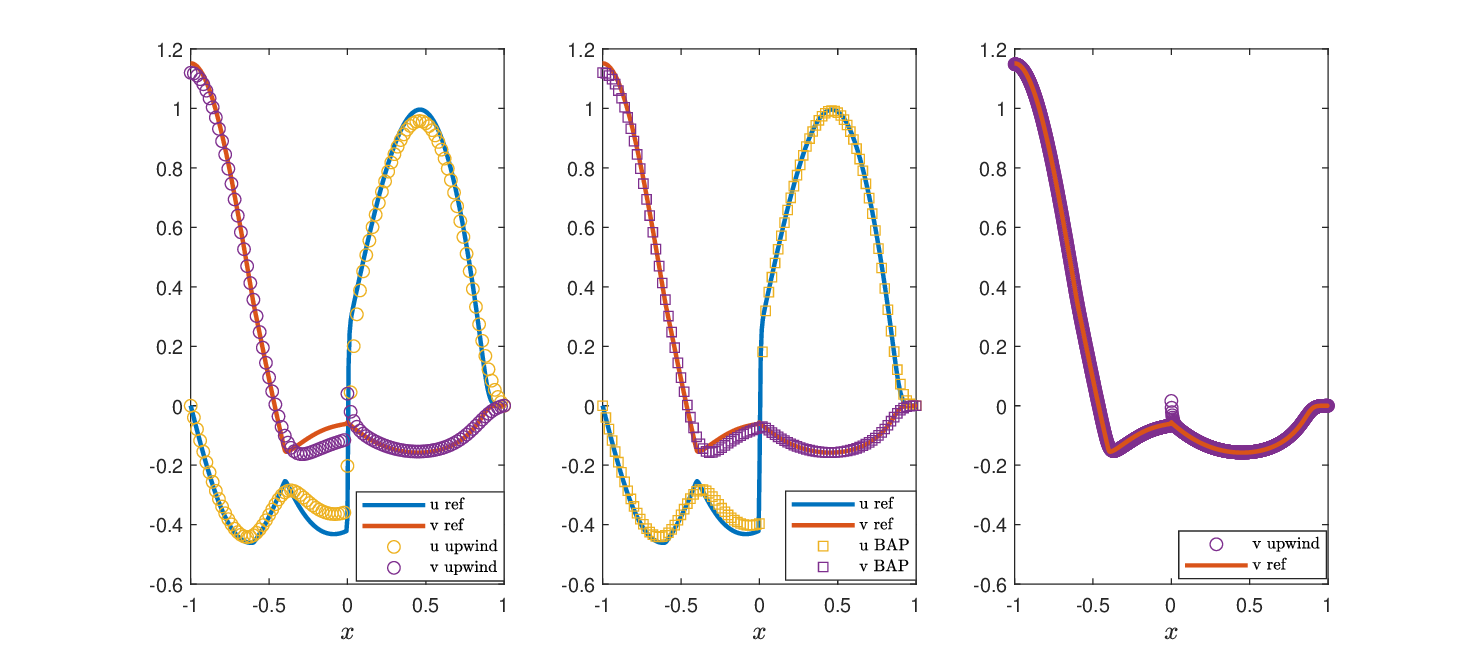}
\caption{(left) the numerical results  by the upwind scheme with $N_x=100$; (middle) the results by our boundary-AP method with $N_x=100$; (right) the results by the upwind scheme with $N_x=1000$.} 
\label{fig:4}
\end{figure}

\textbf{Example 5 (Interface problem for a linear relaxation model):}
In this example, we also build an artificial example to illustrate the necessity of the boundary-AP treatment in dealing with the interface problem.
Consider the relaxation system
\begin{align*}
&
\left\{\begin{array}{l}
\partial_t u-\partial_x u+\partial_x v =0, \\[2mm]
\partial_t v+\partial_x u+\partial_x p  =-\dfrac{1}{\epsilon(x)}v, \\[2mm]
\partial_t p+\partial_x v =-\dfrac{1}{\epsilon(x)}p,
\end{array}\right.\\[2mm]
\text{(initial condition)}\quad&u(x,0)=\sin(\pi x),\quad 
v(x,0)\equiv0,\quad p(x,0)\equiv0.
\end{align*}
The relaxation parameter $\epsilon(x)$ is given by \eqref{Interface-ep} with $\epsilon_0=10^{-6}$. We take the numerical solution solved by the domain decomposition method \cite{MR4593212} as the reference solution. Namely, we solve the relaxation system with $\epsilon=1$ on the left side and solve the equilibrium system $\partial_tu-\partial_xu=0$ on the right side. 

In this example, we take very fine mesh. Namely, we divide the computation domain  $x\in[-1,1]$ into $N_x=2000$ cell. Besides, we also use the classical upwind scheme with the same parameters. The numerical result is shown for $t=0.6$ in Figure \ref{fig:5}. From the figures, we see that the numerical results by our method match the reference solution well. On the other hand, the upwind scheme can not capture the correct interface value and the error may spread. Moreover, we also observe this phenomenon by using the first-order Lax-Friedrichs scheme. The figure is quite similar to that in Figure \ref{fig:5} (right) and we omit it here.
\begin{figure}[h]
\centering
\includegraphics[width=5.0in]{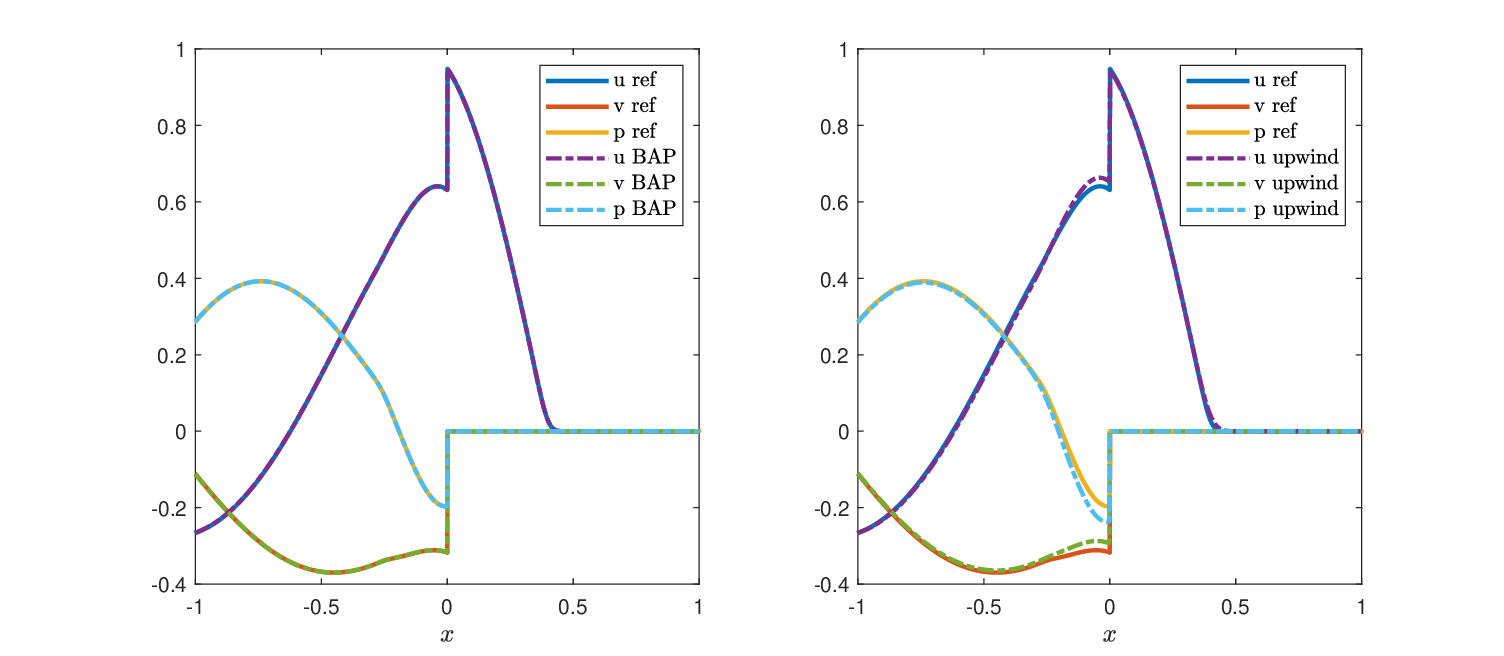}
\caption{The numerical results by our boundary-AP method (left) and by the classical upwind method (right).} 
\label{fig:5}
\end{figure}

\section*{\normalsize{Acknowledgments}} 
The author is funded by Alexander von Humboldt Foundation (Humboldt Research Fellowship Programme for Postdocs). The author would like to thank Prof. Shi Jin for helpful discussion and encouraging him to consider the interface problem.


\begin{appendix}
\section{Missing Proof}\label{AppendA}
\begin{proof}[Proof of Lemma \ref{lem2.1}]
By the definition of $M(\eta)$, we directly compute the eigenvalues 
\begin{small}
$$
\lambda_- = -\frac{1}{2} \Big(a\eta+\sqrt{a^2\eta^2+4(\eta+1)}\Big)<0,\quad \lambda_+ = -\frac{1}{2} \Big(a\eta-\sqrt{a^2\eta^2+4(\eta+1)}\Big)>0
$$
\end{small}
and eigenvectors
\begin{small}
$$
L_-^M = \frac{1}{\sqrt{(1+\eta)^2+\lambda_-^2}}\begin{pmatrix}
\lambda_- &
1+\eta
\end{pmatrix},
\quad L_+^M = \frac{1}{\sqrt{(1+\eta)^2+\lambda_+^2}}\begin{pmatrix}
\lambda_+ &
1+\eta
\end{pmatrix}.
$$
\end{small}
For sufficiently small $\eta$, we have $\lambda_-=-1+O(\eta)$, $\lambda_+=1+O(\eta)$ and 
\begin{small}
$$
L_-^M = \frac{1}{\sqrt{2}}\begin{pmatrix}
-1 & 
1
\end{pmatrix} + O(\eta),\quad 
L_+^M = \frac{1}{\sqrt{2}}\begin{pmatrix}
1 &
1
\end{pmatrix} + O(\eta).
$$
\end{small}

For sufficiently large $\eta$, we discuss according to the sign of $a$. For $a<0$ we compute
\begin{small}
$$
\lambda_- = \frac{2(\eta+1)}{a\eta-\sqrt{a^2\eta^2+4(\eta+1)}}
 = \frac{2(\eta^{-1}+1)}{a-\sqrt{a^2+4\eta^{-1}(1+\eta^{-1})}} = \frac{1}{a}+O(\eta^{-1})
$$
\end{small}
and get
\begin{small}
$$
L_-^M = 
\frac{1}{\sqrt{(1+\eta)^2+\lambda_-^2}}\begin{pmatrix}
\lambda_- &
1+\eta
\end{pmatrix}=
\begin{pmatrix}
0 &
1
\end{pmatrix}+O(\eta^{-1}).
$$
\end{small}
Besides, we compute
\begin{small}
$$
\eta^{-1} \lambda_+  = -\frac{1}{2} \Big(a-\sqrt{a^2+4(\eta^{-1}+\eta^{-2})}\Big) = -a +O(\eta^{-1})
$$
\end{small}
and
\begin{small}
$$
L_+^M = \frac{1}{\sqrt{(1+\eta^{-1})^2+(\eta^{-1}\lambda_+)^2}}\begin{pmatrix}
\eta^{-1} \lambda_+ &
1+\eta^{-1}
\end{pmatrix} = \frac{1}{\sqrt{1+a^2}}\begin{pmatrix}
-a&1
\end{pmatrix}+O(\eta^{-1}).
$$
\end{small}
Notice that $\sqrt{1+a^2}L_+^M$ is also an eigenvector of $M$. 
On the other hand, for $a>0$ we compute
\begin{small}
$$
\lambda_+ = \frac{2(\eta+1)}{a\eta+\sqrt{a^2\eta^2+4(\eta+1)}}
 = \frac{2(\eta^{-1}+1)}{a+\sqrt{a^2+4\eta^{-1}(1+\eta^{-1})}} = \frac{1}{a}+O(\eta^{-1})
$$
\end{small}
and 
\begin{small}
$$
\eta^{-1} \lambda_- =-\frac{1}{2} \Big(a+\sqrt{a^2+4(\eta^{-1}+\eta^{-2})}\Big) = -a+O(\eta^{-1}).
$$ 
\end{small}
Thus we have
\begin{small}
$$
L_-^M = \frac{1}{\sqrt{(1+\eta^{-1})^2+(\eta^{-1}\lambda_-)^2}}\begin{pmatrix}
\eta^{-1}\lambda_- &
1+\eta^{-1}
\end{pmatrix}=\frac{1}{\sqrt{1+a^2}}\begin{pmatrix}
-a&1
\end{pmatrix}+O(\eta^{-1})
$$
\end{small}
and
\begin{small}
$$
L_+^M = \frac{1}{\sqrt{(1+\eta)^2+\lambda_+^2}}\begin{pmatrix}
\lambda_+ &
1+\eta
\end{pmatrix} = \begin{pmatrix}
0 &
1
\end{pmatrix}+O(\eta^{-1}).
$$
\end{small}
This completes the proof.
\end{proof}

\end{appendix}


\begin{thebibliography}{10}

\bibitem{MR1916293}
{\sc G.~Bal and Y.~Maday}, {\em Coupling of transport and diffusion models in linear transport theory}, M2AN Math. Model. Numer. Anal., 36 (2002), pp.~69--86, \url{https://doi.org/10.1051/m2an:2002007}.

\bibitem{MR533346}
{\sc A.~Bensoussan, J.-L. Lions, and G.~C. Papanicolaou}, {\em Boundary layers and homogenization of transport processes}, Publ. Res. Inst. Math. Sci., 15 (1979), pp.~53--157, \url{https://doi.org/10.2977/prims/1195188427}.

\bibitem{MR2284507}
{\sc S.~Benzoni-Gavage and D.~Serre}, {\em Multidimensional hyperbolic partial differential equations}, Oxford Mathematical Monographs, The Clarendon Press, Oxford University Press, Oxford, 2007.
\newblock First-order systems and applications.

\bibitem{MR3816765}
{\sc R.~Borsche and A.~Klar}, {\em Kinetic layers and coupling conditions for scalar equations on networks}, Nonlinearity, 31 (2018), pp.~3512--3541, \url{https://doi.org/10.1088/1361-6544/aabc91}.

\bibitem{MR3867627}
{\sc R.~Borsche and A.~Klar}, {\em A nonlinear discrete velocity relaxation model for traffic flow}, SIAM J. Appl. Math., 78 (2018), pp.~2891--2917, \url{https://doi.org/10.1137/17M1152681}.

\bibitem{MR1262639}
{\sc J.-F. Bourgat, P.~Le~Tallec, B.~Perthame, and Y.~Qiu}, {\em Coupling {B}oltzmann and {E}uler equations without overlapping}, in Domain decomposition methods in science and engineering ({C}omo, 1992), vol.~157 of Contemp. Math., Amer. Math. Soc., Providence, RI, 1994, pp.~377--398, \url{https://doi.org/10.1090/conm/157/01439}.

\bibitem{MR1445737}
{\sc R.~E. Caflisch, S.~Jin, and G.~Russo}, {\em Uniformly accurate schemes for hyperbolic systems with relaxation}, SIAM J. Numer. Anal., 34 (1997), pp.~246--281, \url{https://doi.org/10.1137/S0036142994268090}.

\bibitem{MR3394370}
{\sc Z.~Cai, Y.~Fan, and R.~Li}, {\em A framework on moment model reduction for kinetic equation}, SIAM J. Appl. Math., 75 (2015), pp.~2001--2023, \url{https://doi.org/10.1137/14100110X}.

\bibitem{MR3622624}
{\sc P.~Degond and F.~Deluzet}, {\em Asymptotic-preserving methods and multiscale models for plasma physics}, J. Comput. Phys., 336 (2017), pp.~429--457, \url{https://doi.org/10.1016/j.jcp.2017.02.009}.

\bibitem{MR2674294}
{\sc F.~Filbet and S.~Jin}, {\em A class of asymptotic-preserving schemes for kinetic equations and related problems with stiff sources}, J. Comput. Phys., 229 (2010), pp.~7625--7648, \url{https://doi.org/10.1016/j.jcp.2010.06.017}.

\bibitem{MR416379}
{\sc R.~Gatignol}, {\em Th\'eorie cin\'etique des gaz \`a{} r\'epartition discr\`ete de vitesses}, vol.~Vol. 36 of Lecture Notes in Physics, Springer-Verlag, Berlin-New York, 1975.

\bibitem{MR1618466}
{\sc V.~Giovangigli and M.~Massot}, {\em Asymptotic stability of equilibrium states for multicomponent reactive flows}, Math. Models Methods Appl. Sci., 8 (1998), pp.~251--297, \url{https://doi.org/10.1142/S0218202598000123}.

\bibitem{MR2026400}
{\sc F.~c. Golse, S.~Jin, and C.~D. Levermore}, {\em A domain decomposition analysis for a two-scale linear transport problem}, M2AN Math. Model. Numer. Anal., 37 (2003), pp.~869--892, \url{https://doi.org/10.1051/m2an:2003059}.

\bibitem{MR3645390}
{\sc J.~Hu, S.~Jin, and Q.~Li}, {\em Asymptotic-preserving schemes for multiscale hyperbolic and kinetic equations}, in Handbook of numerical methods for hyperbolic problems, vol.~18 of Handb. Numer. Anal., Elsevier/North-Holland, Amsterdam, 2017, pp.~103--129.

\bibitem{MR4593212}
{\sc J.~Huang, R.~Li, and Y.~Zhou}, {\em Coupling conditions for linear hyperbolic relaxation systems in two-scale problems}, Math. Comp., 92 (2023), pp.~2133--2165, \url{https://doi.org/10.1090/mcom/3845}.

\bibitem{MR1358521}
{\sc S.~Jin}, {\em Runge-{K}utta methods for hyperbolic conservation laws with stiff relaxation terms}, J. Comput. Phys., 122 (1995), pp.~51--67, \url{https://doi.org/10.1006/jcph.1995.1196}.

\bibitem{MR2964096}
{\sc S.~Jin}, {\em Asymptotic preserving ({AP}) schemes for multiscale kinetic and hyperbolic equations: a review}, Riv. Math. Univ. Parma (N.S.), 3 (2012), pp.~177--216.

\bibitem{MR4436589}
{\sc S.~Jin}, {\em Asymptotic-preserving schemes for multiscale physical problems}, Acta Numer., 31 (2022), pp.~415--489, \url{https://doi.org/10.1017/S0962492922000010}, \url{https://doi.org/10.1017/S0962492922000010}.

\bibitem{MR3008837}
{\sc S.~Jin, J.-G. Liu, and L.~Wang}, {\em A domain decomposition method for semilinear hyperbolic systems with two-scale relaxations}, Math. Comp., 82 (2013), pp.~749--779, \url{https://doi.org/10.1090/S0025-5718-2012-02643-3}.

\bibitem{MR1655853}
{\sc S.~Jin, L.~Pareschi, and G.~Toscani}, {\em Diffusive relaxation schemes for multiscale discrete-velocity kinetic equations}, SIAM J. Numer. Anal., 35 (1998), pp.~2405--2439, \url{https://doi.org/10.1137/S0036142997315962}.

\bibitem{MR2573593}
{\sc D.~Jou, J.~Casas-V\'azquez, and G.~Lebon}, {\em Extended irreversible thermodynamics}, Springer, New York, fourth~ed., 2010, \url{https://doi.org/10.1007/978-90-481-3074-0}.

\bibitem{MR678094}
{\sc T.~Kato}, {\em A short introduction to perturbation theory for linear operators}, Springer-Verlag, New York-Berlin, 1982.

\bibitem{MR1335825}
{\sc A.~Klar}, {\em Convergence of alternating domain decomposition schemes for kinetic and aerodynamic equations}, Math. Methods Appl. Sci., 18 (1995), pp.~649--670, \url{https://doi.org/10.1002/mma.1670180806}.

\bibitem{MR1619859}
{\sc A.~Klar}, {\em An asymptotic-induced scheme for nonstationary transport equations in the diffusive limit}, SIAM J. Numer. Anal., 35 (1998), pp.~1073--1094, \url{https://doi.org/10.1137/S0036142996305558}.

\bibitem{MR888058}
{\sc E.~W. Larsen, J.~E. Morel, and W.~F. Miller, Jr.}, {\em Asymptotic solutions of numerical transport problems in optically thick, diffusive regimes}, J. Comput. Phys., 69 (1987), pp.~283--324, \url{https://doi.org/10.1016/0021-9991(87)90170-7}.

\bibitem{MR2460781}
{\sc M.~Lemou and L.~Mieussens}, {\em A new asymptotic preserving scheme based on micro-macro formulation for linear kinetic equations in the diffusion limit}, SIAM J. Sci. Comput., 31 (2008), pp.~334--368, \url{https://doi.org/10.1137/07069479X}.

\bibitem{MR1392419}
{\sc C.~D. Levermore}, {\em Moment closure hierarchies for kinetic theories}, J. Statist. Phys., 83 (1996), pp.~1021--1065, \url{https://doi.org/10.1007/BF02179552}.

\bibitem{MR1774264}
{\sc L.~Mieussens}, {\em Discrete-velocity models and numerical schemes for the {B}oltzmann-{BGK} equation in plane and axisymmetric geometries}, J. Comput. Phys., 162 (2000), pp.~429--466, \url{https://doi.org/10.1006/jcph.2000.6548}.

\bibitem{muller2007history}
{\sc I.~M{\"u}ller}, {\em A history of thermodynamics: the doctrine of energy and entropy}, Springer Science \& Business Media, 2007.

\bibitem{MR4656780}
{\sc M.~Piu, M.~Herty, and G.~Puppo}, {\em Derivation and stability analysis of a macroscopic multilane model for traffic flow}, SIAM J. Appl. Math., 83 (2023), pp.~2052--2072, \url{https://doi.org/10.1137/22M1543288}.

\bibitem{MR1604274}
{\sc W.-C. Wang and Z.~Xin}, {\em Asymptotic limit of initial-boundary value problems for conservation laws with relaxational extensions}, Comm. Pure Appl. Math., 51 (1998), pp.~505--535, \url{https://doi.org/10.1002/(SICI)1097-0312(199805)51:5<505::AID-CPA3>3.0.CO;2-C}.

\bibitem{MR1793199}
{\sc Z.~Xin and W.-Q. Xu}, {\em Stiff well-posedness and asymptotic convergence for a class of linear relaxation systems in a quarter plane}, J. Differential Equations, 167 (2000), pp.~388--437, \url{https://doi.org/10.1006/jdeq.2000.3806}.

\bibitem{MR1900493}
{\sc Z.~Xin and W.-Q. Xu}, {\em Initial-boundary value problem to systems of conservation laws with relaxation}, Quart. Appl. Math., 60 (2002), pp.~251--281, \url{https://doi.org/10.1090/qam/1900493}.

\bibitem{MR1919787}
{\sc W.-Q. Xu}, {\em Initial-boundary value problem for a class of linear relaxation systems in arbitrary space dimensions}, J. Differential Equations, 183 (2002), pp.~462--496, \url{https://doi.org/10.1006/jdeq.2001.4130}.

\bibitem{MR2030150}
{\sc W.-Q. Xu}, {\em Boundary conditions and boundary layers for a multi-dimensional relaxation model}, J. Differential Equations, 197 (2004), pp.~85--117, \url{https://doi.org/10.1016/j.jde.2003.08.007}.

\bibitem{MR1722195}
{\sc W.-A. Yong}, {\em Boundary conditions for hyperbolic systems with stiff source terms}, Indiana Univ. Math. J., 48 (1999), pp.~115--137, \url{https://doi.org/10.1512/iumj.1999.48.1611}.

\bibitem{MR1693210}
{\sc W.-A. Yong}, {\em Singular perturbations of first-order hyperbolic systems with stiff source terms}, J. Differential Equations, 155 (1999), pp.~89--132, \url{https://doi.org/10.1006/jdeq.1998.3584}.

\bibitem{MR4213673}
{\sc Y.~Zhou and W.-A. Yong}, {\em Boundary conditions for hyperbolic relaxation systems with characteristic boundaries of type {I}}, J. Differential Equations, 281 (2021), pp.~289--332, \url{https://doi.org/10.1016/j.jde.2021.02.008}.

\bibitem{MR4355918}
{\sc Y.~Zhou and W.-A. Yong}, {\em Boundary conditions for hyperbolic relaxation systems with characteristic boundaries of type {II}}, J. Differential Equations, 310 (2022), pp.~198--234, \url{https://doi.org/10.1016/j.jde.2021.11.020}.

\bibitem{zhu2015conservation}
{\sc Y.~Zhu, L.~Hong, Z.~Yang, and W.-A. Yong}, {\em Conservation-dissipation formalism of irreversible thermodynamics}, J. Non-Equilib. Thermodyn., 40 (2015), pp.~67--74.

\end{thebibliography}

\end{document}